\documentclass[11pt]{amsart}
\usepackage{amsmath,a4wide,scalerel}
\usepackage{mathabx}
\usepackage{xcolor,graphicx}
\usepackage{mathrsfs}
\usepackage{pgfplots}
\usepackage{subcaption}
\usepackage[normalem]{ulem}
\usepackage{float}
\usepackage{accents}
\usepackage{bbold}
\usepackage[pagewise]{lineno}
\usepackage{hyperref}
\newtheorem{theorem}{Theorem}
\newtheorem{lemma}[theorem]{Lemma}
\newtheorem{proposition}[theorem]{Proposition}
\newtheorem{remark}[theorem]{Remark}
\newtheorem{corollary}[theorem]{Corollary}
\newtheorem{defi}[theorem]{Definition}

\newtheorem{assumption}{Assumption}

\newcommand{\E}{{\mathbb E}}
\newcommand{\PP}{{\mathbb P}}
\newcommand*\diff{\mathop{}\!\mathrm{d}}
\newcommand{\DDN}{D_{N}}
\newcommand{\CDD}{\operatorname{cl} D}
\newcommand{\CDDN}{\operatorname{cl} D_{N}}
\newcommand{\xiN}{\xi_{N}}
\newcommand{\xiNDt}{\hat{\xi}_{N,\Delta t}}
\newcommand{\txiNDt}{\tilde{\xi}_{N,\Delta t}}
\newcommand{\wN}{w_{N}}

\newcommand{\twNDt}{\tilde{w}_{N,\Delta t}}
\newcommand{\EM}{\textrm{EM}}
\newcommand{\SEM}{\textrm{SEM}}
\newcommand{\TE}{\textrm{TE}}
\newcommand{\ABEM}{\textrm{ABEM}}
\newcommand{\ABEP}{\textrm{ABEP}}
\newcommand{\Lip}{\operatorname{Lip}}

\title[]{Artificial barriers for stochastic differential equations and for construction of boundary-preserving schemes}

\date{\today}

\author{Johan Ulander}
              \address{Department of Mathematical Sciences,
              Chalmers University of Technology and University of Gothenburg, 41296~Gothenburg, Sweden}
              \email{\tt johanul@chalmers.se}

\begin{document}

\begin{abstract}
We develop the novel method of artificial barriers for scalar stochastic differential equations (SDEs) and use it to construct boundary-preserving numerical schemes for strong approximation of scalar SDEs, possibly with non-globally Lipschitz drift and diffusion coefficients, whose state-space is either bounded or half-bounded. The idea of artificial barriers is to augment the SDE with artificial barriers outside the state-space to not change the solution process, and then apply a boundary-preserving numerical scheme to the resulting reflected SDE (RSDE). This enables us to construct boundary-preserving numerical schemes that achieve the same strong convergence rate as the corresponding RSDE scheme. Based on the method of artificial barriers, we construct two boundary-preserving schemes that we call the Artificial Barriers Euler--Maruyama (ABEM) scheme and the Artificial Barriers Euler--Peano (ABEP) scheme, respectively. We provide numerical experiments for the ABEM scheme and the numerical results agree with the obtained theoretical results.
\end{abstract}

\maketitle
%\begin{sloppypar}
{\small\noindent


\vspace{1cm}

\section{Introduction}\label{sec:intro}
Stochastic differential equations (SDEs) are used to model phenomena arising in, for example, physics, mathematical finance, mathematical biology and epidemic modelling \cite{GraySIS,Karlin1981ASC,introTostocCalc,MR1214374,MR2001996}. Some SDEs are confined to a strict subset, referred to as the domain of the SDE, of the target-space. Classical schemes for SDEs are known to leave the domains of SDEs \cite{MR4737060}. In this article, we develop the method of artificial barriers and construct, based on this new concept, strongly convergent numerical approximations of solutions of SDEs that are confined to the domain of the considered SDE.

We consider scalar Itô stochastic differential equations (SDEs) of the form
\begin{equation}\label{eq:SDE}
\diff X(t) = f(X(t)) \diff t + g(X(t)) \diff B(t),\ t \in (0,T],
\end{equation}
with some deterministic initial value $X(0) = x_{0}$, where $B$ is a standard Brownian motion and $T \in (0,\infty)$. We assume that the drift coefficient $f$ and diffusion coefficient $g$ are Lipschitz continuous functions on an open set $D = (a,b) \subset \mathbb{R}$ (see Assumption~\ref{ass:Lip}) such that, provided that $x_{0} \in D$, $X(t) \in D$ for every $t \in (0,T]$, almost surely (see Assumption~\ref{ass:Feller}). Observe that we assume that $X(t)$ cannot hit the boundary $\partial D = \{a,b \}$ of the domain, which can be characterised by Feller's classification of unattainable boundaries (see, for example, Chapter $15$ in \cite{Karlin1981ASC}). To cover both bounded and half-bounded cases, we let $a \in \mathbb{R}$ and $b \in \mathbb{R} \cup \{ \infty \}$ be such that $a < b$. We call the case $b = \infty$ the one-finite-boundary case and the case $b \in (a,\infty)$ the two-finite-boundaries case. The case $a = -\infty$ and $b \in (a,\infty)$ is treated in the same way. We do not consider $D = \mathbb{R}$ as the classical Euler–Maruyama (EM) scheme is boundary-preserving in this case. 

We say that a numerical scheme $\hat{X}$ for~\eqref{eq:SDE} is boundary-preserving if $\hat{X}(t) \in D$ for every $t \in [0,T]$, almost surely. Numerical schemes that may hit the boundary, but not exceed them, are referred to as non-strict boundary-preserving. On one hand, in certain applications, the domain of a model has physical significance. Numerical approximations that leave the physical domain of a model can be interpreted as non-physical. Models defined in a domain show up, for example, in the study of population dynamics \cite{GraySIS,KERMACK199133} and of phase-separation dynamics \cite{ALLEN19791085}, defined in $[0,1]$, and in the study of temperature dynamics (for example, the classical heat equation), defined in $[0,\infty)$. On the other hand, a numerical scheme that leaves the physical domain of a model will naturally be affected by forces that do not influence the model. Let us consider an example in the setting of this paper: If the solution of the SDE~\eqref{eq:SDE} only takes values in $D = (0,1)$, then modifying $f$ and $g$ outside $D=(0,1)$ will not impact the solution. Modifying $f$ and $g$ outside $D = (0,1)$ can, however, influence a numerical approximation that can leave the domain $D = (0,1)$. We can avoid such discrepancy between the exact solution and numerical approximations by constructing and studying boundary-preserving numerical schemes.

Let us discuss previous approaches to construct boundary-preserving numerical schemes. Early boundary-preserving schemes were typically implicit schemes. See, for example, \cite{MR3006996,MR2931351,MR2898556,MR3248050,MR1410392} for some works based on implicit techniques. Implicit schemes, however, tend to be computationally slower than explicit ones. The main focus in this work is explicit schemes, but implicit schemes could also be used within the developed framework. Another common approach to construct boundary-preserving schemes is to transform the solution to be approximated to another process that is easier to construct such schemes for. Typical examples of such transformations include the Lamperti transform (see, for example, \cite{Mller2010FromSD}) and the logarithmic transform \cite{MR4274899}. Time splitting methods \cite{MR2341800}, truncation techniques \cite{MR4475995,MR4242953}, and approximation by geometric Brownian motion \cite{MR4177372,Bossy2024} have also successfully been employed to construct boundary-preserving numerical schemes. An issue with the mentioned approaches, except schemes based on truncation, is that they requires regularity assumptions on the coefficient functions. Truncation schemes are the most related schemes to what we proposed in this work. We also mention \cite{MR4780408,MR4729657,MR3433041,MR3331648,DomPres,UweakBP} for some other proposed approaches to construct boundary-preserving schemes.

Let us discuss the Lamperti transform a bit more in detail. Under some assumptions, optimal convergence rates for numerical schemes for SDEs with additive noise can be transferred to these SDEs using the Lamperti transform \cite{MR4220738,MR4737060}. Recently, authors (see \cite{MR4888024,MR4597411}) have constructed higher-order boundary-preserving numerical schemes using the Lamperti transform. Traditionally, numerical schemes utilising the Lamperti transform have combined it with some classical scheme (for example, the EM scheme or the backward EM scheme) for SDEs (see, for example, \cite{MR3006996,MR3248050,MR4274899}). Recently, authors (see \cite{MR4544037,MR4737060}) have combined the Lamperti transform with time splitting methods to construct efficient numerical approximations the solutions to certain types of SDEs. Lamperti-based schemes do, however, require high regularity assumptions (for example, differentiability) on the drift coefficient $f$ and on the diffusion coefficient $g$ and require the existence of an inverse of the Lamperti transform to be well-defined. For instance, $L^{p}(\Omega)$-convergence rate of $1$, for all $p \geq 1$, was obtained in \cite{MR4737060} for an SDE with bounded domain with drift coefficient $f \in C^{2}$ and diffusion coefficient $g \in C^{3}$, and $f$ and $g$ satisfying some technical conditions near the boundary points. See, for example, Theorem~\ref{th:BPschemeConv} for the definition of $L^{p}(\Omega)$-convergence used in this work. In addition, the Lamperti transform is intrinsically problem-specific which makes the implementation of Lamperti-based schemes problem-specific.

The objective of this paper is to construct boundary-preserving schemes for the SDE~\eqref{eq:SDE} that do not depend on the Lamperti transform and for which optimal $L^{p}(\Omega)$-convergence can be obtained with lower regularity assumptions on the drift and diffusion coefficients $f$ and $g$. To this end, we develop modified versions of two classical schemes for reflected stochastic differential equations (RSDEs) to construct non-Lamperti-based boundary-preserving schemes for the SDE~\eqref{eq:SDE}. We refer readers to \cite{Chitashvili,Pilipenko2014AnIT,MR873889,MR4521003,MR529332} for works on the theory of RSDEs and to \cite{MR2416011,MR1341164,MR2689982,MR1341162,MR1357657,MR1840835,SLOMINSKI1994197} for numerical approximations of RSDEs.

The main contributions of the paper are the following:
\begin{itemize}
\item We develop the method of artificial barriers (see Section~\ref{sec:artBarr}) and apply it to construct boundary-preserving approximations of the solution of the SDE in~\eqref{eq:SDE} (see Section~\ref{sec:ABEM} and Section~\ref{sec:ABEP}).
\item We prove that the proposed method of artificial barriers exhibits optimal $L^{p}(\Omega)$-convergence order, for every $p \geq 2$, when combined with an optimal order $L^{p}(\Omega)$-convergent scheme for reflected SDEs. Moreover, we also obtain almost sure pathwise convergence as a corollary (see Corollary~\ref{corr:BPschemeConv}). In particular, we prove that, for Lipschitz continuous coefficient functions, artificial barriers combined with a modified projected EM scheme (referred to as the ABEM scheme) exhibits $L^{p}(\Omega)$-convergence order of $1/2-$ for every $p \geq 2$ (see Theorem~\ref{th:ABEMLpConv}) and combined with a modified Euler--Peano scheme (referred to as the ABEP scheme) exhibits $L^{p}(\Omega)$-convergence order of $1/2$ for every $p \geq 2$ (see Theorem~\ref{th:ABEPLpConv}).
\item We numerically verify boundary-preservation and $L^{2}(\Omega)$-convergence of order $1/2-$ for the ABEM scheme.
\end{itemize}

This paper is organised as follows. We first provide the setting in Section~\ref{sec:setting}. In Section~\ref{sec:SkorMapAndRSDE} we provide a concise introduction to Skorokhod maps and the field of RSDEs. The main contribution of this paper is Section~\ref{sec:artBarr} where we develop the method of artificial barriers and we prove $L^{p}(\Omega)$-convergence. As a corollary, we also obtain almost sure pathwise convergence, see Corollary~\ref{corr:BPschemeConv}. In Section~\ref{sec:ABEM} and Section~\ref{sec:ABEP}, we combine artificial barriers with a modified projected EM scheme (referred to as the ABEM scheme) and with a modified Euler--Peano scheme (referred to as the ABEP scheme), respectively. Lastly, we provide numerical experiments for the ABEM scheme in Section~\ref{sec:numExp} to support our theoretical results in Section~\ref{sec:artBarr} and in Section~\ref{sec:ABEM}.

\section{Setting}\label{sec:setting}
This section introduces the needed notation and assumptions for the construction and study of the method of artificial barriers. We let $(\Omega,\mathcal{F},\mathbb{P},\bigl(\mathcal{F}_t\bigr)_{t\ge 0})$ be a filtered probability space that satisfies the usual conditions and we denote by $\mathbb{E}[\cdot]$ the expectation operator. We denote generic constants depending on parameters $a_{1},\ldots,a_{l}$, for some $l \in \mathbb{N}$, by $C(a_{1},\ldots,a_{l})$ that may vary from line to line. The closure of a set $A$ is denoted by $\operatorname{cl} A$. We let $\mathcal{D}[0,T]$, for some fixed $T \in (0,\infty)$, denote the space of real-valued functions defined on $[0,T]$ that are right continuous with left limits (càdlàg).

Most equalities and inequalities in this work are to be understood in the almost sure sense, we omit specifying this everywhere to avoid reiteration.

We consider scalar time-homogeneous Itô stochastic differential equations (SDEs)
\begin{equation}\label{eq:SDEmain}
\left\lbrace
\begin{aligned}
& \diff X(t) = f(X(t)) \diff t + g(X(t)) \diff B(t),\ t \in (0,T], \\ 
& X(0) = x_{0} \in D,
\end{aligned}
\right.
\end{equation}
where $T \in (0,\infty)$, $f,g : \mathbb{R} \to \mathbb{R}$ are functions to be specified below and $D \subset \mathbb{R}$ is an open set such that
\begin{equation*}
\PP (X(t) \in D,\ \forall t \in [0,T]) = 1.
\end{equation*}
Note that $D$ depends on the choices of $f$ and $g$. Let $a \in \mathbb{R}$ and $b \in \mathbb{R} \cup \{ \infty \}$ be such that $a < 0 < b$, and be the boundary points of $D$: $\partial D = \{ a,b \}$. The other cases are treated similarly. We say that a stochastic process $\left( X(t) \right)_{t \in [0,T]}$ is a (strong) solution of \eqref{eq:SDEmain} if the corresponding Itô integral equation
\begin{equation*}
X(t) = x_{0} + \int_{0}^{t} f(X(s)) \diff s + \int_{0}^{t} g(X(s)) \diff B(s)
\end{equation*}
is satisfied, for every $t \in [0,T]$, almost surely, where the second integral is an Itô integral. We refer to \cite{MR2001996} for a well-know textbook on Itô SDEs.

We next state the assumptions on the considered SDE that are needed for the method of artificial barriers to be well-defined and for the convergence analysis. We need the following regularity assumptions on $f$ and $g$ to obtain the desired $L^{p}(\Omega)$-convergence rate in Section~\ref{sec:artBarr}.

\begin{assumption}\label{ass:Lip}
The drift coefficient $f: D \to \mathbb{R}$ and diffusion coefficient $g: D \to \mathbb{R}$ are Lipschitz continuous functions with Lipschitz constants $\Lip(f)$ and $\Lip(g)$, respectively; that is, there exists $\Lip(f), \Lip(f)>0$ such that
\begin{equation*}
|f(x) - f(y)| \leq \Lip(f) |x-y|,\ \forall x,y \in D,
\end{equation*}
and 
\begin{equation*}
|g(x) - g(y)| \leq \Lip(g) |x-y|,\ \forall x,y \in D.
\end{equation*}
\end{assumption}
\begin{remark}
Assumption~\ref{ass:Lip} only assumes that the drift and diffusion coefficient functions are Lipschitz continuous functions on the domain $D$. If $|D| < \infty$, then $f$ and $g$ (view as functions $\mathbb{R} \to \mathbb{R}$) can be locally Lipschitz functions with superlinear growth. See Section~\ref{sec:numExp} for a possible choice of superlinearly growing $f$ and $g$.
\end{remark}
Note that Assumption~\ref{ass:Lip} implies, in particular, that
\begin{equation}\label{eq:fLG}
|f(x)| \leq |f(0)| + \Lip(f) |x|,\ \forall x \in D,
\end{equation}
and
\begin{equation}\label{eq:gLG}
|g(x)| \leq |g(0)| + \Lip(g) |x|,\ \forall x \in D.
\end{equation}
We introduce
\begin{equation*}
L_{f} = \max \left( |f(0)|, \Lip(f) \right),
\end{equation*}
to have one common bound for $\Lip(f)$ and for the linear growth constant in~\eqref{eq:fLG}. For the same reason, we let
\begin{equation*}
L_{g} = \max \left( |g(0)|, \Lip(g) \right).
\end{equation*}

The following assumption is crucial for us, as it motivates the study of boundary-preserving schemes for the SDE~\eqref{eq:SDEmain}. Assumption~\ref{ass:Feller} and similar assumptions are sometimes referred to as Feller conditions. We refer to Chapter $15$ in \cite{Karlin1981ASC} for details on Feller's classification of unattainable boundaries for SDEs.
\begin{assumption}\label{ass:Feller}
The drift coefficient $f: D \to \mathbb{R}$ and diffusion coefficient $g: D \to \mathbb{R}$ are such that the boundary points $a$ and $b$ are unattainable by solutions $X$ of \eqref{eq:SDEmain}; that is, we assume that $\PP (X(t) \in D,\ \forall t \in [0,T]) = 1$.
\end{assumption}

We state well-posedness of the problem~\eqref{eq:SDEmain} under Assumptions~\ref{ass:Lip} and~\ref{ass:Feller} in the following proposition, the proof follows from standard theory for SDEs with globally Lipschitz coefficients (by modifying $f$ and $g$ outside $D$). See, for example, \cite{MR2001996} for a classical book on SDEs. Note that if $|b|<\infty$, then $|X(t)| \leq \max(|a|,|b|)$ for all $t \in [0,T]$, almost surely, by Assumption~\ref{ass:Feller}.
\begin{proposition}\label{prop:xproperties}
If Assumptions~\ref{ass:Lip} and~\ref{ass:Feller} are satisfied, then there exists a unique continuous (strong) solution $X$ to the SDE~\eqref{eq:SDEmain} with finite moments
\begin{equation*}
\E \left[ \sup_{t \in [0,T]} |X(t)|^{p} \right] \leq
\begin{cases}
C(p,L_{f},L_{g},x_{0}), \text{ if } b = \infty,\\
\max(|a|^{p},|b|^{p}), \text{ if } |b| < \infty,
\end{cases}
\end{equation*}
for every $p \geq 2$, and satisfying
\begin{equation*}
\PP \left( X(t) \in D,\ \forall t \in [0,T] \right)= 1.
\end{equation*}
\end{proposition}

\section{Skorokhod maps and Reflected SDEs}\label{sec:SkorMapAndRSDE}
Here we introduce Skorokhod maps and define reflected SDEs (RSDEs) in terms of Skorokhod maps. Shortly put, Skorokhod maps (in our case) are constructed to map any path in $\mathcal{D}[0,T]$ to a path in $\mathcal{D}[0,T]$ confined to a subset of $\mathbb{R}$. RSDEs, also known as SDEs with boundary conditions, are essentially SDEs that are only allowed to take values in a certain subset of $\mathbb{R}$, if the process hits the boundary of the subset then the process is reflected. 

\subsection{Skorokhod maps}\label{sec:SkorMap}
First, we introduce and provide some properties of Skorokhod maps underlying the theory of RSDEs discussed in the next section. In essence, a Skorokhod map takes, in our case, a function in $\mathcal{D}[0,T]$ and outputs another function in $\mathcal{D}[0,T]$ that is confined to the domain under consideration. We say that a Skorokhod map $\Gamma$ is associated with a domain $A$ if $\Gamma (\psi)(t) \in A$ for all $\psi \in \mathcal{D}[0,T]$ and for all $t \in [0,T]$. We refer to Skorokhod maps associated with half-bounded domains $[\alpha,\infty)$, for $\alpha \in \mathbb{R}$, as one-sided Skorokhod maps and we refer to Skorokhod maps associated with bounded domains $[\alpha,\beta]$, for $\alpha < \beta$ and $|\alpha|+|\beta|<\infty$, as two-sided Skorokhod maps. One-sided Skorokhod maps were first considered in \cite{MR4521003}. Moreover, an explicit formula for one-sided Skorokhod maps (see Definition~\ref{def:Gexplicit} below) has been known since long (see, for example, \cite{SLOMINSKI1994197}). An explicit formula for two-sided Skorokhod maps (see Definition~\ref{def:G0aExplicit} below) was only later derived (see \cite{MR2349573}). We are interested in Skorokhod maps for the purpose of constructing implementable numerical schemes and for proving $L^{p}(\Omega)$-convergence. Note, however, that explicit formulas for Skorokhod maps do not guarantee that numerical schemes based on Skorokhod maps can be implemented (see, for example, the ABEP scheme in Section~\ref{sec:ABEP}). Certain schemes based on Skorokhod maps can, however, be implemented. Most importantly, some piecewise constant schemes combined with a Skorokhod map can be implemented (see, for example, the ABEM scheme in Section~\ref{sec:ABEM}).

As our interest is only the one-dimensional case, we define the Skorokhod maps using the known explicit representations. See, for example, \cite{MR529332} for a definition of the Skorokhod map for a convex domain in $d$-dimensional Euclidean space.

\begin{defi}\label{def:Gexplicit}
The one-sided Skorokhod map $\Gamma^{0} : \mathcal{D}[0,T] \to \mathcal{D}[0,T]$ associated with the domain $[0,\infty)$ is the map defined by
\begin{equation}\label{eq:G0explicit}
\Gamma^{0}(\psi)(t) = \psi(t) + \sup_{s  \in [0,t]} \max \{0, - \psi(s) \},\ \psi \in \mathcal{D}[0,T],\ t \in [0,T].
\end{equation}
The one-sided Skorokhod map $\Gamma^{\alpha} : \mathcal{D}[0,T] \to \mathcal{D}[0,T]$, $\alpha \neq 0$, associated with the domain $[\alpha,\infty)$ is defined by
\begin{equation}\label{eq:Gaexplicit}
\Gamma^{\alpha}(\psi)(t) = \Gamma^{0}(\psi - \alpha)(t) + \alpha,\ \psi \in \mathcal{D}[0,T],\ t \in [0,T].
\end{equation}
\end{defi}

We provide two useful properties of the Skorokhod map $\Gamma^{\alpha}$ in the following lemma.
\begin{lemma}\label{lem:GaLipGB}
The one-sided Skorokhod map $\Gamma^{\alpha}: \mathcal{D}[0,T] \to \mathcal{D}[0,T]$, for $\alpha \in \mathbb{R}$, satisfies the following Lipschitz property in the supremum norm
\begin{equation}\label{eq:GaLip}
\sup_{t \in [0,r]} |\Gamma^{\alpha}(\psi_{1})(t)-\Gamma^{\alpha}(\psi_{2})(t)| \leq 2 \sup_{t \in [0,r]} |\psi_{1}(t)-\psi_{2}(t)|,
\end{equation}
for every $\psi_{1},\ \psi_{2} \in \mathcal{D}[0,T]$ and for every $r \in [0,T]$, and the following growth bound in the supremum norm
\begin{equation}\label{eq:GaGB}
\sup_{t \in [0,r]} |\Gamma^{\alpha}(\psi)(t)| \leq 2 \sup_{t \in [0,r]} |\psi(t)| + 3 |\alpha|,
\end{equation}
for every $\psi \in \mathcal{D}[0,T]$ and for every $r \in [0,T]$.
\end{lemma}
Note that the proof of Lemma~\ref{lem:GaLipGB} reduces to the case $\alpha = 0$ by~\eqref{eq:Gaexplicit}. We refer to \cite{MR1110990} for a proof of (a more general version in the space of real-valued continuous functions on $[0,T]$) \eqref{eq:GaLip} and the estimate in equation~\eqref{eq:GaGB} with $\alpha=0$ follows directly from the representation for $\Gamma^{0}$ in equation~\eqref{eq:G0explicit}. Note that the properties in Lemma~\ref{lem:GaLipGB} (and corresponding lemmas below) are important for the properties of RSDEs in the following section (see Proposition~\ref{prop:RSDEwellDef}).

We next introduce the two-sided Skorokhod maps associated with bounded domains.
\begin{defi}\label{def:G0aExplicit}
The two-sided Skorokhod map $\Gamma^{0,\alpha} : \mathcal{D}[0,T] \to \mathcal{D}[0,T]$, for $\alpha>0$, associated with the domain $[0,\alpha]$ is the map defined by
\begin{equation}\label{eq:G0aExplicit}
\Gamma^{0,\alpha} = \Lambda^{\alpha} \circ \Gamma^{0},
\end{equation}
where $\Gamma^{0}$ is the one-sided Skorokhod map associated with the domain $[0,\infty)$ defined in equation~\eqref{eq:G0explicit} and where the map $\Lambda^{\alpha}: \mathcal{D}[0,T] \to \mathcal{D}[0,T]$ is defined by
\begin{equation*}
\Lambda^{\alpha}(\psi)(t) = \psi(t) - \sup_{s \in [0,t]} \left( \min \left( \max(0,\psi(s)-\alpha) , \inf_{u \in [s,t]} \psi(u) \right) \right),\ \psi \in \mathcal{D}[0,T]. 
\end{equation*}
\end{defi}
We refer to \cite{MR2349573} for more details on~\eqref{eq:G0aExplicit}.

Similarly to $\Gamma^{\alpha}$, $\Gamma^{0,\alpha}$ satisfies a Lipschitz estimate in the supremum norm and a bound in the supremum norm that we state in the following lemma.
\begin{lemma}\label{lem:G0aProp}
The two-sided Skorokhod map $\Gamma^{0,\alpha}: \mathcal{D}[0,T] \to \mathcal{D}[0,T]$, for $\alpha>0$, satisfies the following Lipschitz property in the supremum norm
\begin{equation}\label{eq:G0aLip}
\sup_{t \in [0,r]} |\Gamma^{0,\alpha}(\psi_{1})(t)-\Gamma^{0,\alpha}(\psi_{2})(t)| \leq 4 \sup_{t \in [0,r]} |\psi_{1}(t)-\psi_{2}(t)|,
\end{equation}
for every $\psi_{1},\ \psi_{2} \in \mathcal{D}[0,T]$ and for every $r \in [0,T]$, and is bounded in the supremum norm
\begin{equation}\label{eq:G0aGB}
\sup_{t \in [0,r]} |\Gamma^{0,\alpha}(\psi)(t)| \leq \alpha,
\end{equation}
for every $\psi \in \mathcal{D}[0,T]$ and for every $r \in [0,T]$.
\end{lemma}
We refer to \cite{MR2349573} for a proof of the estimate in~\eqref{eq:G0aLip} and the estimate in~\eqref{eq:G0aGB} is immediate since $\Gamma^{0,\alpha}(\psi)(t) \in [0,\alpha]$ for every $t \in [0,T]$.

Next, we introduce the Skorokhod maps needed for the method of artificial barriers developed in Section~\ref{sec:artBarr}. To this end, we first introduce a family of domains $\{ \DDN \}_{N \in \mathbb{N}}$: For $N \in \mathbb{N}$, let
\begin{equation*}
\DDN=
\begin{cases}
(a + 1/N, \infty), \text{ if } b = \infty,\\
(a + 1/N, b - 1/N), \text{ if } |b| < \infty.
\end{cases}
\end{equation*}
We are interested in the Skorokhod map $\Gamma$ associated with $\CDD = (a,b) \cup \{ a,b \}$ and the family of Skorokhod maps $\{ \Gamma_{N} \}_{N \in \mathbb{N}}$ each associated with $\CDDN = \partial \DDN \cup \DDN$. The maps $\Gamma$ and $\Gamma_{N}$ can be expressed in terms of $\Gamma^{\alpha}$ and $\Gamma^{0,\alpha}$ as
\begin{equation}\label{eq:GabExplicit}
\Gamma(\psi)(t)=
\begin{cases}
\Gamma^{a}(\psi)(t), \text{ if } b = \infty,\\
\Gamma^{0,(b-a)}(\psi - a)(t) + a, \text{ if } |b| < \infty,
\end{cases}
\end{equation}
and
\begin{equation}\label{eq:GNexplicit}
\Gamma_{N}(\psi)(t)=
\begin{cases}
\Gamma^{a + 1/N}(\psi)(t), \text{ if } b = \infty,\\
\Gamma^{0,(b-a-2/N)}\left(\psi - (a + 1/N) \right)(t) + (a + 1/N), \text{ if } |b| < \infty,
\end{cases}
\end{equation}
both for $\psi \in \mathcal{D}[0,T]$ and $t \in [0,T]$. By~\eqref{eq:GabExplicit}, Lemma~\ref{lem:GaLipGB}, and Lemma~\ref{lem:G0aProp}, $\Gamma$ satisfies the Lipschitz estimate 
\begin{equation*}
\sup_{t \in [0,r]} |\Gamma(\psi_{1})(t)-\Gamma(\psi_{2})(t)| \leq 4 \sup_{t \in [0,r]} |\psi_{1}(t)-\psi_{2}(t)|,
\end{equation*}
for $\psi_{1},\ \psi_{2} \in \mathcal{D}[0,T]$ and for $r \in [0,T]$, and the growth bound
\begin{equation*}
\sup_{t \in [0,r]} |\Gamma(\psi)(t)| \leq
\begin{cases}
2 \sup_{t \in [0,r]} |\psi(t)| + 3 |a|, \text{ if } b = \infty,\\
\max(|a|,|b|), \text{ if } |b| < \infty,
\end{cases}
\end{equation*}
for $\psi \in \mathcal{D}[0,T]$ and $r \in [0,T]$. The following lemma states that $\{ \Gamma_{N} \}_{N \in \mathbb{N}}$ satisfies Lipschitz bounds and growth estimates, uniformly in $N \in \mathbb{N}$.
\begin{lemma}\label{lem:GNLipGB}
The family of Skorokhod maps $\{ \Gamma_{N} \}_{N \in \mathbb{N}}$ satisfies
\begin{equation}\label{eq:GNLip}
\sup_{N \in \mathbb{N}} \sup_{t \in [0,r]} |\Gamma_{N}(\psi_{1})(t)-\Gamma_{N}(\psi_{2})(t)| \leq 4 \sup_{t \in [0,r]} |\psi_{1}(t)-\psi_{2}(t)|,
\end{equation}
for all $\psi_{1},\ \psi_{2} \in \mathcal{D}[0,T]$ and for $r \in [0,T]$, and
\begin{equation}\label{eq:GNGB}
\sup_{N \in \mathbb{N}} \sup_{t \in [0,r]} |\Gamma_{N}(\psi)(t)| \leq
\begin{cases}
2 \sup_{t \in [0,r]} |\psi(t)| + 3 |a| + 3, \text{ if } b = \infty,\\
\max(|a|,|b|), \text{ if } |b| < \infty,
\end{cases}
\end{equation}
for all $\psi \in \mathcal{D}[0,T]$ and for $r \in [0,T]$.
\end{lemma}

\begin{proof}[Proof of Lemma~\ref{lem:GNLipGB}]
Suppose first that $|b|<\infty$. The growth bound in~\eqref{eq:GNGB} is immediate since $\Gamma_{N}(\psi)(t) \in \CDDN \subset \CDD = [a,b]$ for all $t \in [0,T]$. The Lipschitz bound in~\eqref{eq:GNLip} follows from reducing it to the known estimates in Lemma~\ref{lem:G0aProp} using~\eqref{eq:GNexplicit} for $|b|<\infty$
\begin{align*}
\sup_{t \in [0,r]} |\Gamma_{N}(\psi_{1})(t)-\Gamma_{N}(\psi_{2})(t)| &= \sup_{t \in [0,r]} |\Gamma^{0,(b-a-2/N)}(\psi_{1})(t)-\Gamma^{0,(b-a-2/N)}(\psi_{2})(t)| \\ &\leq 4 \sup_{t \in [0,r]} |\psi_{1}(t)-\psi_{2}(t)|.
\end{align*}
The Lipschitz bound for the case $b=\infty$ is proved similarly:
\begin{align*}
\sup_{t \in [0,r]} |\Gamma_{N}(\psi_{1})(t)-\Gamma_{N}(\psi_{2})(t)| &= \sup_{t \in [0,r]} |\Gamma^{a + 1/N}(\psi_{1})(t)-\Gamma^{a + 1/N}(\psi_{2})(t)| \\ &\leq 2 \sup_{t \in [0,r]} |\psi_{1}(t)-\psi_{2}(t)| \\ &\leq 4 \sup_{t \in [0,r]} |\psi_{1}(t)-\psi_{2}(t)|
\end{align*}
using~\eqref{eq:GNexplicit} for $b=\infty$ and Lemma~\ref{lem:GaLipGB}. The growth bound for $b = \infty$ follows from~\eqref{eq:GaGB}
\begin{equation*}
\sup_{t \in [0,r]} |\Gamma_{N}(\psi)(t)| = \sup_{t \in [0,r]} |\Gamma^{a + 1/N}(\psi)(t)| \leq 2 \sup_{t \in [0,r]} |\psi(t)| + 3 \left|a + \frac{1}{N} \right| \leq 2 \sup_{t \in [0,r]} |\psi(t)| + 3 |a| + 3.
\end{equation*}
\end{proof}
In the following sections, we will not separate the one-finite-boundary case and the two-finite-boundaries case and we will apply Lemma~\ref{lem:GNLipGB} to both cases.

\subsection{Reflected SDEs}\label{sec:RSDE}
Next, we provide an overview of reflected SDEs (RSDEs), we refer to \cite{Pilipenko2014AnIT} for a more complete introduction to RSDEs. We define RSDEs for the Skorokhod map $\Gamma$ associated with $\CDD = [a,b]$, for $|a|+|b|<\infty$, and the other considered domains are treated analogously. We first provide a short intuitive description and the idea of RSDEs, and then we provide a proper definition. In this section, instead of imposing the assumptions in Section~\ref{sec:setting}, we assume that $f,g: \mathbb{R} \to \mathbb{R}$ are globally Lipschitz functions.

For simplicity, we start with the case of one barrier at $0$ ($a=0$ and $b=\infty$) with continuous sample paths. The goal in this case is to define a stochastic process $\xi$ with the following three properties:
\begin{enumerate}
\item $\xi(t) \geq 0$ for all $t \in [0,T]$ and $t \mapsto \xi(t)$ is continuous, almost surely.
\item $\xi$ has dynamics according to the Itô SDE
\begin{equation*}
\diff \xi(t) = f(\xi(t)) \diff t + g(\xi(t)) \diff B(t)
\end{equation*}
whenever $\xi(t) >0$.
\item $\xi$ is reflected back into the domain $(0,\infty)$ whenever $\xi$ hits $0$ to avoid $\xi$ entering $(-\infty,0)$.
\end{enumerate}
Note that, as there is no mentioning of the reflection in (2) above, the reflection dynamics in (3) must disappear as soon as $\xi$ re-enters $(0,\infty)$.  The above problem was first studied by Skorokhod in \cite{MR4521003} who proposed to formalise the problem, thereafter known as the Skorokhod problem, as finding a pair of non-anticipating processes $\xi(t) \geq 0$ and $L(t)$, for $t \in [0,T]$, such that
\begin{equation}\label{eq:xiDef}
\xi(t) = \xi_{0} + \int_{0}^{t} f(\xi(s)) \diff s + \int_{0}^{t} g(\xi(s)) \diff B(s) + L(t),
\end{equation}
for all $t \in [0,T]$, almost surely, $L$ is non-decreasing, $L(0)=0$ and
\begin{equation}\label{eq:Ldef}
L(t) = \int_{0}^{t} \mathbb{1}_{\xi(s)=0} \diff L(s),\ t \in [0,T].
\end{equation}
Note that the property~\eqref{eq:Ldef} states that the process $L(t)$ can only increase when $\xi(t) = 0$ and that $L(t)$ then acts in~\eqref{eq:xiDef} as a positive driving force to push $\xi$ into $(0,\infty)$. We use the letter $L$ to emphasise that it represent a lower barrier process. Moreover, under Lipschitz conditions on $f$ and $g$, there exists a unique solution to the Skorokhod problem. We refer to \cite{Pilipenko2014AnIT} for more details. The considered setting in Section~\ref{sec:setting} reduces to the Lipschitz case by Assumption~\ref{ass:Feller}.

Next, we provide the proper mathematical definition of RSDEs with two barriers at $a$ and $b$ (with $|a|+|b|<\infty$).
\begin{defi}\label{def:RSDE}
A triplet of $\mathcal{F}_{t}$-adapted processes $(\xi(t),U(t),L(t))_{t \in [0,T]}$ is a solution of the reflected SDE
\begin{equation*}
\diff \xi(t) = \xi_{0} + f(\xi(t)) \diff t + g(\xi(t)) \diff B(t) - \diff U(t) + \diff L(t),\ t \in [0,T],
\end{equation*}
with reflection at $a$ and $b$, with $-\infty<a<b<\infty$, and the initial condition $\xi(0) = \xi_{0} \in [a,b]$ if the following holds:
\begin{enumerate}
\item $\xi(t) \in [a,b]$ for every $t \in [0,T]$, almost surely.
\item $t \mapsto U(t)$ and $t \mapsto L(t)$ are non-decreasing and $U(0) = L(0)=0$.
\item $\int_{0}^{t} \mathbb{1}_{\xi(s) \in (\infty,b)} \diff U(s) = \int_{0}^{t} \mathbb{1}_{\xi(s) \in (a,\infty)} \diff L(s) = 0$, for every $t \in [0,T]$.
\item $\xi(t) = \xi(0) + \int_{0}^{t} f(\xi(s)) \diff s + \int_{0}^{t} g(\xi(s)) \diff B(s) - U(t) + L(t)$, for every $t \in [0,T]$, almost surely.
\item All the above integrals are well-defined.
\end{enumerate}
\end{defi}
Note that if $w$ solves the Itô functional SDE
\begin{equation}\label{eq:wItoSDE}
w(t) = \xi_{0} + \int_{0}^{t} f(\Gamma(w)(s)) \diff s + \int_{0}^{t} g(\Gamma(w)(s)) \diff B(s),\ t \in [0,T],
\end{equation}
then there exist $U(t)$ and $L(t)$ such that $(\Gamma(w)(t),U(t),L(t))_{t \in [0,T]}$ satisfy Definition~\ref{def:RSDE} (see, for example, Section $1.2$ in \cite{MR2349573}), where we recall that $\Gamma$ is the Skorokhod map associated with $[a,b]$ (see Section~\ref{sec:SkorMap}). We say that $(\xi,w,L-U)$ is an associated triple on $[a,b]$ if Definition~\ref{def:RSDE} and~\eqref{eq:wItoSDE} holds.
\begin{remark}
If $\xi$ cannot hit one of the boundary points (for example $b$) then the corresponding barrier process (in this example $U$) can be completely removed or set equal to $0$ (in this example $U \equiv 0$) in the definition above. This follows from uniqueness of solutions to RSDEs, see below or \cite{Pilipenko2014AnIT} for more details. Of particular interest in this paper is the case where $\xi$ cannot hit either of the boundary points (see Assumption~\ref{ass:Feller}), which implies that both barrier processes will, in fact, be $0$ (meaning that $U \equiv 0 \equiv L$), this enables us to connect the original SDE in~\eqref{eq:SDE} to RSDEs (see Proposition~\ref{prop:xNxiNequal}).
\end{remark}

We now collect some well-known results for RSDEs that will be needed for the convergence analysis. We refer to \cite{Pilipenko2014AnIT} for a proof, but the existence and uniqueness is due to \cite{MR4521003}.
\begin{proposition}\label{prop:RSDEwellDef}
If $f,g: \mathbb{R} \to \mathbb{R}$ in Definition~\ref{def:RSDE} are globally Lipschitz continuous, then there exists a unique continuous solution $\xi$ to the RSDE in Definition~\ref{def:RSDE}. Moreover, $\xi$ satisfies moment bounds
\begin{equation*}
\E \left[ \sup_{t \in [0,T]} |\xi(t)|^{p} \right] \leq
\begin{cases}
C(p,T,L_{f},L_{g},\xi_{0}), \text{ if } b = \infty,\\
\max(|a|^{p},|b|^{p}), \text{ if } |b| < \infty,
\end{cases}
\end{equation*}
for every $p \geq 2$, and the time increments satisfy
\begin{equation*}
\E \left[ \left| \xi(t) - \xi(s) \right|^{p} \right] \leq C(p,T,L_{f},L_{g},\xi_{0}) |t-s|^{p/2},
\end{equation*}
for every $p \geq 2$, for every $t,s \in [0,T]$.
\end{proposition}

\section{Method of artificial barriers}\label{sec:artBarr}
We now present the idea of adding artificial barriers to the types of SDEs in equation~\eqref{eq:SDEmain} and we call it the method of artificial barriers. Note that Assumption~\ref{ass:Feller} is important for this methodology to work and be of interest. 

\subsection{Construction of artificial barriers}
The method of artificial barriers for constructing boundary-preserving numerical schemes is based on the idea that if the solution process $X$ to an SDE cannot leave a domain, then we can augment the SDE with barriers outside this domain to obtain an RSDE whose solution equals the original process $X$, almost surely. The obtained RSDE can then be approximated using any boundary-preserving numerical scheme for RSDEs. To straight away apply this idea to the considered setting would be problematic because numerical schemes for RSDEs are in general non-strict boundary-preserving; that is, the boundary points are not unattainable by numerical schemes for RSDEs. This is an issue since $f$ and $g$ that we consider typically vanish on the boundary, and most numerical schemes that hit the boundary will thereafter stay at the boundary. Instead, shortly put, we modify the coefficient functions $f$ and $g$ to move the boundary slightly inside the original domain $D$ and augment the modified SDE with artificial barriers. We can then discretise the resulting (modified) RSDE using a (non-strict) boundary-preserving numerical scheme. Note that the exact details of adding artificial barriers depend on the specific problem at hand, the here outlined method is constructed to fit our purposes. The rest of this section is dedicated to formulate the method of artificial barriers mathematically.

The first step of the method of artificial barriers is to modify the drift and diffusion coefficients such that the unattainable boundary points are moved to $a + 1/N$ and $b - 1/N$. If $b = \infty$, then we interpret $b - 1/N = \infty$. To formalise this, we introduce an integer $N \in \mathbb{N}$ and consider the following modified Itô SDE
\begin{equation}\label{eq:modSDE}
\left\lbrace
\begin{aligned}
& \diff X_{N}(t) = f_{N}(X_{N}(t)) \diff t + g_{N}(X_{N}(t)) \diff B(t),\ t \in (0,T], \\ 
& X_{N}(0) = x_{0} \in \DDN,
\end{aligned}
\right.
\end{equation}
where $\DDN = (a + 1/N,b-1/N) \subset D$ and where $f_{N}, g_{N} : \DDN \to \mathbb{R}$ are defined below in such a way that $X_{N}(t) \in \DDN$, for every $t \in [0,T]$, almost surely. For ease of presentation and to avoid technicalities regarding the initial value, we assume that $x_{0} \in D_{1}$. $x_{0} \in D_{1}$ can be relaxed to $x_{0} \in D$ by defining $X_{N}(0) = \Pi_{N-1}(x_{0}) \in \DDN$, where $\Pi_{N-1}: \mathbb{R} \to \operatorname{cl} D_{N-1} \subset \DDN$ is the projection onto $\operatorname{cl} D_{N-1}$ (see~\eqref{eq:PiNdef} for a definition). Replacing $x_{0}$ with $\Pi_{N-1}(x_{0})$ in~\eqref{eq:modSDE} would only introduce a deterministic error of order $1/N$, and would not alter the convergence rates. 

We specify the functions $f_{N}$ and $g_{N}$ as follows: For $x \in \DDN$, let
\begin{equation}\label{eq:psiDef}
\Psi_{N}(x) = 
\begin{cases}
x - \frac{1}{N}, \text{ if } b = \infty,\\
\frac{b-a}{b-a-\frac{2}{N}} \left( x - \left( a + \frac{1}{N} \right) \right) + a, \text{ if } |b| < \infty,
\end{cases}
\end{equation}
\begin{equation*}
f_{N}(x) = f \left( \Psi_{N}(x) \right),
\end{equation*}
and
\begin{equation*}
g_{N}(x) = g \left( \Psi_{N}(x) \right).
\end{equation*}
As $x \mapsto \Psi_{N}(x)$ in~\eqref{eq:psiDef} is a linear function (both for $b=\infty$ and for $|b|<\infty$), the decay of $f_{N}$ near $a + 1/N$ and near $b - 1/N$ is the same as the decay of $f$ near $a$ and near $b$, respectively. The same is true for $g_{N}$ and $g$. This implies that the boundary behaviour of $X$ (see Proposition~\ref{prop:xproperties}) is the same as the boundary behaviour of $X_{N}$ (see Proposition~\ref{prop:xiNproperties} below). Observe that the case $b=\infty$ in~\eqref{eq:psiDef} can be obtained from the case $|b|<\infty$ in~\eqref{eq:psiDef} by
\begin{equation*}
\lim_{b \to \infty} \Psi_{N}(x) = \lim_{b \to \infty} \frac{b-a}{b-a-\frac{2}{N}} \left( x - \left( a + \frac{1}{N} \right) \right) + a = x - \frac{1}{N}.
\end{equation*}
\begin{remark}
The above is only well-defined if $N > \frac{2}{b-a}$. If $|b-a| \leq 2$, then we simply restrict $N \in \mathbb{N}$ to $N > \frac{2}{b-a}$. Henceforth, we assume, for ease of presentation, that $|b-a| > 2$.
\end{remark}

The following lemma states that the modified coefficient functions $f_{N}$ and $g_{N}$ satisfies a Lipschitz estimate on $D$ uniformly in $N \in \mathbb{N}$.
\begin{lemma}\label{lem:fNgNLip}
If Assumption~\ref{ass:Lip} is satisfied, then there exist constants $\hat{L}_{f}$ and $\hat{L}_{g}$ such that the following Lipschitz conditions hold
\begin{equation*}
\left| f_{N}(x) - f_{N}(y) \right| \leq \hat{L}_{f} |x-y|,\ \forall x,y \in \DDN,
\end{equation*}
and
\begin{equation*}
\left| g_{N}(x) - g_{N}(y) \right| \leq \hat{L}_{g} |x-y|,\ \forall x,y \in \DDN,
\end{equation*}
for every $N \in \mathbb{N}$. Moreover, $\hat{L}_{f}$ and $\hat{L}_{g}$ can be chosen as
\begin{equation*}
\hat{L}_{f} =
\begin{cases}
\Lip(f), \text{ if } b = \infty,\\
\Lip(f) \left| \frac{b-a}{b-a-2} \right|, \text{ if } |b| < \infty,
\end{cases}
\end{equation*}
and
\begin{equation*}
\hat{L}_{g} =
\begin{cases}
\Lip(g), \text{ if } b = \infty,\\
\Lip(g) \left| \frac{b-a}{b-a-2} \right|, \text{ if } |b| < \infty.
\end{cases}
\end{equation*}
\end{lemma}
and such that the following linear growth conditions hold
\begin{equation*}
|f_{N}(x)| \leq \hat{L}_{f}(1 + |x|),\ x \in \DDN,
\end{equation*}
and
\begin{equation*}
|g_{N}(x)| \leq \hat{L}_{g}(1 + |x|),\ x \in \DDN,
\end{equation*}
for every $N \in \mathbb{N}$. 
\begin{proof}[Proof of Lemma~\ref{lem:fNgNLip}]
We provide the proof for $f_{N}$ in the case $|b|<\infty$, the other cases are proved in the same way. We use Lipschitz continuity of $f$ on $D$ and the defining formula for $\Psi_{N}$ in~\eqref{eq:psiDef} to see that
\begin{align*}
\left| f_{N}(x) - f_{N}(y) \right| &= \left| f(\Psi_{N}(x)) - f(\Psi_{N}(y)) \right| \leq \Lip(f) \left| \Psi_{N}(x) - \Psi_{N}(y) \right| \\ &= \Lip(f) \left| \frac{b-a}{b-a-2/N} \right| |x-y| \\ &\leq \Lip(f) \left| \frac{b-a}{b-a-2} \right| |x-y|.
\end{align*}
\end{proof}
Alternatively, as the composition of two Lipschitz continuous functions is a Lipschitz continuous function, Lemma~\ref{lem:fNgNLip} also follows from
\begin{equation}\label{eq:psiLip}
\operatorname{Lip}(\Psi_{N}) = \sup_{x \in \mathbb{R}} \left| \Psi_{N}'(x) \right| = 
\begin{cases}
1, \text{ if } b = \infty,\\
\left| \frac{b-a}{b-a-\frac{2}{N}} \right|, \text{ if } |b| < \infty,
\end{cases}
\leq
\begin{cases}
1, \text{ if } b = \infty,\\
\left| \frac{b-a}{b-a-2} \right|, \text{ if } |b| < \infty.
\end{cases}
\end{equation}

The following lemma quantifies the error introduced by replacing $f$ and $g$ with $f_{N}$ and $g_{N}$, respectively, which will be used for the comparison between $X$ and $X_{N}$ in Proposition~\ref{prop:ModSDELpConv} below.
\begin{lemma}\label{lem:fNgNdiff}
If Assumption~\ref{ass:Lip} is satisfied, then there exist constants $C(L_{f},a,b)$ and $C(L_{g},a,b)$ such that
\begin{equation*}
\left| f(x) - f_{N}(x) \right| \leq C(L_{f},a,b) N^{-1} (1 + |x|),\ x \in D,
\end{equation*}
\begin{equation*}
\left| g(x) - g_{N}(x) \right| \leq C(L_{g},a,b) N^{-1} (1 + |x|),\ x \in D,
\end{equation*}
for every $N \in \mathbb{N}$. If $b=\infty$, then the constants are independent of $a$, $b$, and $x$.
\end{lemma}
\begin{proof}[Proof of Lemma~\ref{lem:fNgNdiff}]
We prove the estimate for $f$ and $f_{N}$ in the case $|b|<\infty$, the other cases are proved in the same way. By definition of $f_{N}$, we have that
\begin{equation*}
|f(x) - f_{N}(x)| = |f(x)-f(\Psi_{N}(x))| \leq L_{f} |x - \Psi_{N}(x)|.
\end{equation*}
Moreover, by definition of $\Psi_{N}$ in~\eqref{eq:psiDef} we have that
\begin{equation}\label{eq:psiEst}
  \begin{split}
    \left| x - \Psi_{N}(x) \right| &= \left| \frac{b + a}{N (b-a) - 2} - \frac{2}{N (b-a) - 2} x\right| \\ &\leq N^{-1} \left( \frac{|b+a|}{|b-a-2/N|} + \frac{2}{|b-a-2/N|} |x| \right) \\
    &\leq N^{-1} \left( \frac{|b+a|}{|b-a-2|} + \frac{2}{|b-a-2|} |x| \right) \leq N^{-1} C(a,b) (1 + |x|)
  \end{split}
\end{equation}
and inserting this into the estimate for $|f(x) - f_{N}(x)|$ above gives the desired result. Note that the estimate~\eqref{eq:psiEst} reduces to
\begin{equation*}
\left| x - \Psi_{N}(x) \right| = \frac{1}{N}
\end{equation*}
in the case $b = \infty$ and the constants in the lemma can be chosen to be independent of $a$, $b$, and $x$.
\end{proof}
As a result of Lemma~\ref{lem:fNgNLip} and Lemma~\ref{lem:fNgNdiff}, we can bound the linear growth constant for $f_{N}$, independently of $N \in \mathbb{N}$:
\begin{equation}\label{eq:fNLG}
|f_{N}(x)| \leq |f_{N}(x)-f(x)| + |f(x)| \leq C(L_{f},a,b) (1 + |x| ),\ x \in D.
\end{equation}
The same holds for $g_{N}$.

From the above lemmas implies, in particular, that the domain of~\eqref{eq:modSDE} is $\DDN$ and its solution $X_{N}$ enjoys moment bounds, uniform in $N \in \mathbb{N}$.
\begin{proposition}\label{prop:xNproperties}
Let $N \in \mathbb{N}$. If Assumptions~\ref{ass:Lip} and~\ref{ass:Feller} are satisfied, then there exists a unique solution $X_{N}$ to the SDE~\eqref{eq:modSDE} with modified coefficients that satisfies moments bounds
\begin{equation*}
\E \left[ \sup_{t \in [0,T]} |X_{N}(t)|^{p} \right] \leq 
\begin{cases}
C(p,T,L_{f},L_{g},x_{0}), \text{ if } b = \infty,\\
\max(|a|^{p},|b|^{p}), \text{ if } |b| < \infty,
\end{cases}
\end{equation*}
for every $p \geq 2$, and that only takes values in $\DDN$
\begin{equation*}
\PP \left( X_{N}(t) \in \DDN,\ \forall t \in [0,T] \right)= 1.
\end{equation*}
Moreover, the constant $C(p,T,L_{f},L_{g},x_{0})$ is independent of $N \in \mathbb{N}$.
\end{proposition}
\begin{proof}[Proof of Proposition~\ref{prop:xNproperties}]
Existence and uniqueness of a solution follows from Proposition~\ref{prop:xproperties} (up to modifying $f_{N}$ and $g_{N}$ outside $\DDN$). By~\eqref{eq:fNLG}, the moment bound constant is independent of $N \in \mathbb{N}$. As $\Psi_{N}$ defined in~\eqref{eq:psiDef} is a linear function, the local behaviour of $X_{N}$ near $\partial \DDN$ is the same as the local behaviour of $X$ near $\partial D$. Thus, 
\begin{equation*}
\PP \left( X_{N}(t) \in \DDN,\ \forall t \in [0,T] \right)= 1
\end{equation*}
follows from Assumption~\ref{ass:Feller}.
\end{proof}

The proof of convergence $X_{N} \to X$ in $L^{p}(\Omega)$, as $N \to \infty$, in Proposition~\ref{prop:ModSDELpConv} below uses the two following lemmas and Grönwall's inequality.
\begin{lemma}\label{lem:IntffNerr}
Let $p \geq 1$ and let $t \in [0,T]$. Suppose that Assumption~\ref{ass:Lip} is satisfied and let $X$ and $Y$ be $\mathcal{F}_{t}$-adapted stochastic processes only taking values in $D$. Then
\begin{equation*}
\left| \int_{0}^{t} f(X(s)) - f_{N}(Y(s)) \diff s \right|^{p} \leq C(p,T,L_{f},a,b) \left( \int_{0}^{t} |X(s) - Y(s)|^{p} \diff s + N^{-p} \left(1 + \sup_{s \in [0,t]} |Y(s)|^{p} \right) \right).
\end{equation*}
In particular,
\begin{align*}
\E \left[ \sup_{t \in [0,r]} \left| \int_{0}^{t} f(X(s)) - f_{N}(Y(s)) \diff s \right|^{p} \right] &\leq C(p,T,L_{f},a,b) \int_{0}^{t} \E \left[ \sup_{t \in [0,s]} |X(t) - Y(t)|^{p} \right] \diff s \\ &+  C(p,T,L_{f},a,b) N^{-p} \left(1 + \E \left[ \sup_{t \in [0,T]} |Y(t)|^{p} \right] \right).
\end{align*}
If $b=\infty$, then the constants are independent of $a$ and $b$.
\end{lemma}
\begin{remark}
The estimates in Lemma~\ref{lem:IntffNerr} are also true with $\sup_{s \in [0,t]} |Y(s)|^{p}$ replaced with $\sup_{s \in [0,t]} |X(s)|^{p}$.
\end{remark}
The proof of Lemma~\ref{lem:IntffNerr} is standard, we include it for completeness.
\begin{proof}[Proof of Lemma~\ref{lem:IntffNerr}]
Let $p \geq 1$. We split the error and apply Jensen's inequality for integrals to estimate
\begin{align*}
\left| \int_{0}^{t} f(X(s)) - f_{N}(Y(s)) \diff s \right|^{p} &\leq C(p,T) \left( \int_{0}^{t} |f(X(s)) - f(Y(s))|^{p} + |f(Y(s)) - f_{N}(Y(s))|^{p} \diff s \right) \\ &\leq C(p,T,L_{f},a,b) \left( \int_{0}^{t} |X(s) - Y(s)|^{p} \diff s + N^{-p} \left( 1 + \sup_{s \in [0,t]} |Y(s)|^{p} \right) \right),
\end{align*}
where we in the last inequality used that $f$ is Lipschitz continuous on $D$ and Lemma~\ref{lem:fNgNdiff} for the second term. This concludes the proof.
\end{proof}

\begin{lemma}\label{lem:IntggNerr}
Let $p \geq 2$ and let $r \in [0,T]$. Suppose Assumption~\ref{ass:Lip} is satisfied and let $X$ and $Y$ be $\mathcal{F}_{t}$-adapted stochastic processes only taking values in $D$. Then
\begin{multline*}
\E \left[ \sup_{t \in [0,r]} \left| \int_{0}^{t} g(X(s)) - g_{N}(Y(s)) \diff B(s) \right|^{p} \right] \\ \leq C(p,T,L_{g},a,b) \left( \int_{0}^{r} \E \left[ |X(s) - Y(s)|^{p} \right] \diff s + N^{-p} \left(1 + \E \left[ \sup_{s \in [0,r]} |Y(s)|^{p} \right] \right) \right).
\end{multline*}
In particular,
\begin{multline*}
\E \left[ \sup_{t \in [0,r]} \left| \int_{0}^{t} g(X(s)) - g_{N}(Y(s)) \diff B(s) \right|^{p} \right] \\ \leq C(p,T,L_{g},a,b) \left( \int_{0}^{r} \E \left[ \sup_{t \in [0,s]} |X(t) - Y(t)|^{p} \right] \diff s + N^{-p} \left(1 + \E \left[ \sup_{t \in [0,T]} |Y(t)|^{p} \right] \right) \right).
\end{multline*}
\end{lemma}
If $b=\infty$, then the constants are independent of $a$ and $b$.
\begin{remark}
In the same way as for Lemma~\ref{lem:IntffNerr}, the estimates in Lemma~\ref{lem:IntggNerr} are also true with $\sup_{s \in [0,t]} |Y(s)|^{p}$ replaced with $\sup_{s \in [0,t]} |X(s)|^{p}$.
\end{remark}
The proof of Lemma~\ref{lem:IntggNerr} is also standard.
\begin{proof}[Proof of Lemma~\ref{lem:IntggNerr}]
Let $p \geq 2$. We apply the BDG inequality and Jensen's inequality for integrals (requires $p \geq 2$) to see that
\begin{equation*}
\E \left[ \sup_{t \in [0,r]} \left| \int_{0}^{t} g(X(s)) - g_{N}(Y(s)) \diff B(s) \right|^{p} \right] \leq C(p,T) \E \left[ \int_{0}^{r} \left| g(X(s)) - g_{N}(Y(s)) \right|^{p} \diff s \right]
\end{equation*}
and the same argument as we used in the proof of Lemma~\ref{lem:IntffNerr} gives us
\begin{align*}
\int_{0}^{r} \left| g(X(s)) - g_{N}(Y(s)) \right|^{p} \diff s &\leq C(p,L_{g},a,b) \left( \int_{0}^{r} |X(s) - Y(s)|^{p} \diff s  + \int_{0}^{r} |g(Y(s)) - g_{N}(Y(s))|^{p} \diff s \right) \\ &\leq C(p,L_{g},a,b) \left( \int_{0}^{r} |X(s) - Y(s)|^{p} \diff s  + N^{-p} \left( 1 + \sup_{s \in [0,r]} |Y(s)|^{p} \right) \right),
\end{align*}
where we also used Lemma~\ref{lem:fNgNdiff}. Thus, we have arrived at the desired estimate
\begin{multline*}
\E \left[ \sup_{t \in [0,r]} \int_{0}^{t} \left| g(X(s)) - g_{N}(Y(s)) \right|^{p} \diff s \right] \\ \leq C(p,T,L_{g},a,b) \left( \int_{0}^{r} \E \left[ |X(s) - Y(s)|^{p} \right] \diff s  + N^{-p} \left( 1 + \E \left[ \sup_{s \in [0,r]} |Y(s)|^{p} \right] \right) \right).
\end{multline*}
\end{proof}

Lemma~\ref{lem:IntffNerr} and Lemma~\ref{lem:IntggNerr} are the main ingredients in the following proposition.
\begin{proposition}\label{prop:ModSDELpConv}
If Assumptions~\ref{ass:Lip} and~\ref{ass:Feller} are satisfied, then for every $p \geq 2$ there exists a constant $C(p,T,L_{f},L_{g},a,b,x_{0})$ such that
\begin{equation*}
\E \left[ \sup_{t \in [0,T]} |X_{N}(t) - X(t)|^{p} \right] \leq C(p,T,L_{f},L_{g},a,b,x_{0}) N^{-p}
\end{equation*}
for every $N \in \mathbb{N}$ with $x_{0} \in \DDN$. If $b=\infty$, then the constant is independent of $a$ and $b$.
\end{proposition}
\begin{proof}[Proof of Proposition~\ref{prop:ModSDELpConv}]
Let $p \geq 2$. The proof uses Lemma~\ref{lem:IntffNerr} and Lemma~\ref{lem:IntggNerr} combined with Grönwall's inequality. To this end, let $r \in [0,T]$. We use the integral equations for $X$ and $X_{N}$ in equations~\eqref{eq:SDE} and~\eqref{eq:modSDE} to estimate
\begin{align*}
\sup_{t \in [0,r]} \left| X_{N}(t) - X(t) \right|^{p} \leq & C(p) \sup_{t \in [0,r]} \left| \int_{0}^{t} f(X(s)) - f_{N}(X_{N}(s)) \diff s \right|^{p} \\&+ C(p) \sup_{t \in [0,r]} \left| \int_{0}^{t} g(X(s))) - g_{N}(X_{N}(s)) \diff B(s) \right|^{p}
\end{align*}
for some constant $C(p)$ depending on $p$. We now apply Lemma~\ref{lem:IntffNerr} and Lemma~\ref{lem:IntggNerr} to the expected value of the first and second terms, respectively, on the right hand side of the above inequality. We obtain that
\begin{multline*}
\E \left[ \sup_{t \in [0,r]} \left| X_{N}(t) - X(t) \right|^{p} \right] \\ \leq C(p,T,L_{f},L_{g},a,b) \left( \int_{0}^{r} \E \left[ \sup_{t \in [0,s]} \left| X(t) - X_{N}(t) \right|^{p}  \right] \diff s + N^{-p} \left( 1 + \E \left[ \sup_{s \in [0,T]} |X_{N}(s)|^{p} \right] \right) \right).
\end{multline*}
Moment bounds for $X_{N}$ in Proposition~\ref{prop:xNproperties} and Grönwall's inequality now gives the desired result
\begin{equation*}
\E \left[ \sup_{t \in [0,T]} \left| X_{N}(t) - X(t) \right|^{p} \right] \leq C(p,T,L_{f},L_{g},a,b,x_{0}) N^{-p},
\end{equation*}
and the constant is independent of $a$ and $b$ if $b=\infty$.
\end{proof}

The second step of the method of artificial barriers is to augment the SDE~\eqref{eq:modSDE} with what we call artificial barriers. To this end, we consider the following reflected process $\xi_{N}$ (see Section~\ref{sec:RSDE} for details) that can be viewed as $X_{N}$ augmented with artificial barriers:
\begin{equation}\label{eq:RmodSDE}
\left\lbrace
\begin{aligned}
& \diff \xi_{N}(t) = f_{N}(\xi_{N}(t)) \diff t + g_{N}(\xi_{N}(t)) \diff B(t) - \diff U_{N}(t) + \diff L_{N}(t),\ t \in (0,T], \\ 
& \xi_{N}(0) = x_{0} \in \DDN,
\end{aligned}
\right.
\end{equation}
where (alternative intuitive formulation of Definition~\ref{def:RSDE})
\begin{equation*}
\diff U_{N}(t) = \mathbb{1}_{b - 1/N}(\xi_{N}(t)) \diff U_{N}(t),
\end{equation*}
with initial value $U_{N}(0)=0$, is the upper barrier at $b - 1/N$ and where
\begin{equation*}
\diff L_{N}(t) = \mathbb{1}_{a + 1/N}(\xi_{N}(t)) \diff L_{N}(t),
\end{equation*}
with initial value $L_{N}(0)=0$, is the lower barrier at $a + 1/N$. Recall the integral equation for $\xi_{N}$
\begin{equation*}
\xi_{N}(t) = x_{0} + \int_{0}^{t} f_{N}(\xi_{N}(s)) \diff s + \int_{0}^{t} g_{N}(\xi_{N}(s)) \diff B(s) - U_{N}(t) + L_{N}(t),\ t \in [0,T],
\end{equation*}
and that $(\xi_{N},w_{N},L_{N}-U_{N})$, where
\begin{equation}\label{eq:wNdef}
w_{N}(t) = x_{0} + \int_{0}^{t} f_{N}(\xi_{N}(s)) \diff s + \int_{0}^{t} g_{N}(\xi_{N}(s)) \diff B(s),\ t \in [0,T],
\end{equation}
is an associated triple in the sense of Definition~\ref{def:RSDE}. See Section~\ref{sec:RSDE} for more details on RSDEs. 
\begin{proposition}\label{prop:xiNproperties}
Suppose Assumptions~\ref{ass:Lip} and~\ref{ass:Feller} are satisfied and let $N \in \mathbb{N}$. Then there exists a unique solution $\xi_{N}$ to the RSDE~\eqref{eq:modSDE} with modified coefficients with finite moments
\begin{equation}\label{eq:xNmomBound}
\E \left[ \sup_{t \in [0,T]} |\xi_{N}(t)|^{p} \right] \leq
\begin{cases}
C(p,T,L_{f},L_{g},x_{0}), \text{ if } b = \infty,\\
\max(|a|^{p},|b|^{p}), \text{ if } |b| < \infty,
\end{cases}
\end{equation}
for every $p \geq 2$, finite moments of time increments
\begin{equation}\label{eq:xNtimeReg}
\E \left[ \left| \xi_{N}(t) - \xi_{N}(s) \right|^{p} \right] \leq C(p,T,L_{f},L_{g},a,b,x_{0}) |t-s|^{p/2},
\end{equation}
for every $p \geq 2$, and that only takes values in $\CDDN$
\begin{equation*}
\PP \left( \xi_{N}(t) \in \CDDN,\ \forall t \in [0,T] \right)= 1.
\end{equation*}
Moreover, the constants are independent of $N \in \mathbb{N}$. If $b=\infty$, then the constant for the time increments is independent of $a$ and $b$.
\end{proposition}
\begin{proof}[Proof of Proposition~\ref{prop:xiNproperties}]
Existence and uniqueness of solutions follows from Proposition~\ref{prop:RSDEwellDef} with globally Lipschitz coefficients (up to modifying $f_{N}$ and $g_{N}$ outside $\DDN$). By Lemma~\ref{lem:fNgNLip}, the constants in~\eqref{eq:xNmomBound} and in~\eqref{eq:xNtimeReg} are independent of $N \in \mathbb{N}$ and independent of $a$ and $b$ if $b=\infty$. $\xi_{N}(t) \in \CDDN$ for every $t \in [0,T]$, almost surely, is immediate from the definition of solutions to RSDEs.
\end{proof}
The following proposition is the reason we call the processes $U_{N}$ and $L_{N}$ artificial barriers, as their inclusion have no impact on the process $X_{N}$. 
\begin{proposition}\label{prop:xNxiNequal}
Suppose Assumptions~\ref{ass:Lip} and~\ref{ass:Feller} are satisfied and let $N \in \mathbb{N}$. Then the solution $X_{N}$ of~\eqref{eq:modSDE} and the solution $\xi_{N}$ of~\eqref{eq:RmodSDE} are, almost surely, equal:
\begin{equation*}
\PP \left( X_{N}(t) = \xi_{N}(t),\ \forall t \in [0,T] \right)=1.
\end{equation*}
\end{proposition}
\begin{proof}[Proof of Proposition~\ref{prop:xNxiNequal}]
If $X_{N}$ is a solution of~\eqref{eq:modSDE} and $X_{N} \in D_{N}$, then $(X_{N},\ U_{N} \equiv 0, L_{N} \equiv 0)$ is a solution to~\eqref{eq:RmodSDE}. By uniqueness of solutions of~\eqref{eq:RmodSDE}, we conclude that $X_{N}$ and $\xi_{N}$ are equal, almost surely.
\end{proof}

The third step of the method of artificial barriers is to discretise the RSDE~\eqref{eq:RmodSDE} using any (non-strict) boundary-preserving numerical scheme for RSDEs. To this end, we introduce a discretisation parameter $M \in \mathbb{N}$ and we discretise the time interval $[0,T]$ into $M \in \mathbb{N}$ intervals $[t_{m},t_{m+1}]$ each of size $\Delta t = T/M$. We also let $\ell_{\Delta t}(s) = t_{m}$ for $s \in [t_{m},t_{m+1})$ that will be used in Section~\ref{sec:ABEM} and in Section~\ref{sec:ABEP}. Then, in particular, $t_{m} = m \Delta t$ for every $m = 0,\ldots, M$.  For now, to complete the convergence statement in this section, we assume that $\hat{\xi}_{N,\Delta t}$ is a (non-strict) boundary-preserving scheme for the RSDE in~\eqref{eq:RmodSDE} with error constant independent of $N \in \mathbb{N}$.
\begin{assumption}\label{ass:BPRSDEscheme}
Let $N \in \mathbb{N}$ and $\Delta t >0$. For every $x_{0} \in D_{1}$, there exists an approximation $\hat{\xi}_{N,\Delta t}$ to the solution $\xi_{N}$ of the RSDE in~\eqref{eq:RmodSDE} that only takes values in $\CDDN$
\begin{equation*}
\PP \left( \hat{\xi}_{N,\Delta t}(t) \in \CDDN,\ \forall t \in [0,T] \right)=1,
\end{equation*}
has bounded moments
\begin{equation*}
\E \left[ \sup_{t \in [0,T]} |\hat{\xi}_{N,\Delta t}(t)|^{p} \right] \leq
\begin{cases}
C(p,T,L_{f},L_{g},a,x_{0}), \text{ if } b = \infty,\\
\max(|a|^{p},|b|^{p}), \text{ if } |b| < \infty,
\end{cases}
\end{equation*}
for every $p \geq 2$, and with $L^{p}(\Omega)$-error
\begin{equation*}
\E \left[ \sup_{t \in [0,T]} |\hat{\xi}_{N,\Delta t}(t) - \xi_{N}(t)|^{p} \right] \leq C(p,T,L_{f},L_{g},a,b,x_{0}) \left( N^{-p} + \Delta t^{\gamma p} \right),
\end{equation*}
for every $p \geq 2$ and for some $\gamma>0$, where the constants are independent of $N$ and $\Delta t$. If $b=\infty$, then the constant for the $L^{p}(\Omega)$-error is independent of $b$.
\end{assumption}
In the next sections we define the Artificial Barriers EM (ABEM) scheme (Section~\ref{sec:ABEM}) and the Artificial Barriers Euler--Peano (ABEP) scheme (Section~\ref{sec:ABEP}), and we verify that both satisfy Assumption~\ref{ass:BPRSDEscheme}. The ABEM and ABEP schemes combine artificial barriers with a modified version of the classical projected EM scheme for RSDEs and with a modified version of the classical Euler--Peano scheme for RSDEs, respectively. Note that the method of artificial barriers can be combined with other boundary-preserving schemes for RSDEs (for example, the projected and reflected scheme, see \cite{MR1341162}). We now state the main result of this section.
\begin{theorem}\label{th:BPschemeConv}
Let $N \in \mathbb{N}$ and let $\Delta t >0$. Suppose Assumptions~\ref{ass:Lip} and~\ref{ass:Feller} are satisfied and that $\hat{\xi}_{N,\Delta t}$ satisfies Assumption~\ref{ass:BPRSDEscheme} for some $\gamma>0$. Then $\hat{\xi}_{N,\Delta t}$ is boundary-preserving
\begin{equation*}
\PP \left( \hat{\xi}_{N,\Delta t}(t) \in D,\ \forall t \in [0,T] \right)=1
\end{equation*}
and
\begin{equation*}
\E \left[ \sup_{t \in [0,T]} |\hat{\xi}_{N,\Delta t}(t) - X(t)|^{p} \right] \leq C(p,T,L_{f},L_{g},a,b,x_{0}) \left( N^{-p} + \Delta t^{\gamma p} \right),
\end{equation*}
for every $p \geq 2$, where the constant is independent of $N$ and $\Delta t$. If $b=\infty$, then the constant is independent of $b$.
\end{theorem}
\begin{proof}[Proof of Theorem~\ref{th:BPschemeConv}]
Let $p \geq 2$. The first statement follows from $\CDDN \subsetneq D$, for every $N \in \mathbb{N}$, and from Assumption~\ref{ass:BPRSDEscheme}. Recall the definitions of $X_{N}$ and $\xi_{N}$ in equation~\eqref{eq:modSDE} and in equation~\eqref{eq:RmodSDE}, respectively. Then, by the triangle inequality, we have that
\begin{align*}
\left( \E \left[ \sup_{t \in [0,T]} |\hat{\xi}_{N,\Delta t}(t) - X(t)|^{p} \right] \right)^{\frac{1}{p}} &\leq \left( \E \left[ \sup_{t \in [0,T]} |\hat{\xi}_{N,\Delta t}(t) - \xi_{N}(t)|^{p} \right] \right)^{\frac{1}{p}} \\ &+ \left( \E \left[ \sup_{t \in [0,T]} |X_{N}(t) - X(t)|^{p} \right] \right)^{\frac{1}{p}},
\end{align*}
where we also used that $\xi_{N}$ and $X_{N}$ are equal almost surely (see Proposition~\ref{prop:xNxiNequal}). The result now follows from Assumption~\ref{ass:BPRSDEscheme} and  Proposition~\ref{prop:ModSDELpConv}.
\end{proof}

We now couple $N \in \mathbb{N}$ and $\Delta t$ to obtain, as a corollary to Theorem~\ref{th:BPschemeConv}, almost sure pathwise convergence of $\hat{\xi}_{N,\Delta t}$ to the solution $X$ of the SDE~\eqref{eq:SDEmain}.

\begin{corollary}\label{corr:BPschemeConv}
Suppose Assumptions~\ref{ass:Lip} and~\ref{ass:Feller} are satisfied and suppose that $\hat{\xi}_{N,\Delta t}$ satisfies Assumption~\ref{ass:BPRSDEscheme} for some $\gamma>0$, $N \in \mathbb{N}$ and $\Delta t = N^{-1/\gamma}$. Then there exists for every $\epsilon>0$ a random variable $\eta_{\epsilon}$, with $\E \left[ \eta_{\epsilon}^{p} \right] < \infty$ for every $p \geq 1$, such that
\begin{equation*}
\sup_{t \in [0,T]} |\hat{\xi}_{N,\Delta t}(t) - X(t)| \leq \eta_{\epsilon} \Delta t^{\gamma - \epsilon},
\end{equation*}
almost surely. Moreover, $\eta_{\epsilon}$ is independent of $N \in \mathbb{N}$ and $\Delta t$.
\end{corollary}
The proof of Corollary~\ref{corr:ABEMasConv} follows from applying Lemma~$2.1$ in \cite{MR2320830} to Theorem~\ref{th:BPschemeConv}.

\section{Artificial Barriers Euler--Maruyama}\label{sec:ABEM}
In this section, we combine the method of artificial barriers from Section~\ref{sec:artBarr} with a modification of the classical projected Euler--Maruyama (PEM) scheme for RSDEs applied to the RSDE with modified coefficients in~\eqref{eq:RmodSDE}. We call the resulting scheme the Artificial Barriers EM scheme and abbreviate it as ABEM. We first introduce the ABEM scheme, and then we show that the ABEM scheme satisfy Assumption~\ref{ass:BPRSDEscheme} in Section~\ref{sec:artBarr}. We can from this conclude that the scheme is boundary-preserving and converges to the solution $X$ of the considered SDE~\eqref{eq:SDE}, by Theorem~\ref{th:BPschemeConv}.
We remark that the standard projected EM (PEM) scheme applied to the SDE in~\eqref{eq:SDE} is non-strict boundary-preserving, meaning that the approximations might hit the boundary $\partial D$. A non-strict boundary-preserving numerical scheme is particularly problematic in our considered cases where both $f$ and $g$ typically have zeros at the boundary, since such a scheme would forever remain at the boundary once it hits it.

$L^{p}(\Omega)$-convergence of order $1/2-$ of the PEM scheme for RSDEs is a known result in the literature (see, for example, \cite{MR1357657} for the case $p=2$). We show in this section that this is also true for the ABEM scheme for the RSDEs in~\eqref{eq:RmodSDE} and that the constants for the moment bounds and the $L^{p}(\Omega)$-error can be taken to be independent of $N$ and $\Delta t$, and hence the ABEM scheme satisfy Assumption~\ref{ass:BPRSDEscheme}. 

Let $\Pi_{N} : \mathbb{R} \to \CDDN$ denote the projection onto $\CDDN$, defined by
\begin{equation}\label{eq:PiNdef}
\Pi_{N}(x)=
\begin{cases}
a + 1/N,\text{ for } x < a + 1/N, \\
x,\text{ for } x \in \CDDN, \\
b - 1/N,\text{ for } x > b - 1/N.
\end{cases}
\end{equation}
If $b=\infty$, then we interpret $x > b - 1/N$ in~\eqref{eq:PiNdef} as never true. We define the ABEM scheme, denoted by $\hat{\xi}^{\ABEM}_{N,\Delta t}$, as follows: Let $\hat{\xi}^{\ABEM}_{N,\Delta t}(0) = x_{0} \in D_{1}$  and successively define $\hat{\xi}^{\ABEM}_{N,\Delta t}(t_{m+1})$, for $m=0,\ldots, M-1$, as
\begin{equation}\label{eq:ABEMdef}
\hat{\xi}^{\ABEM}_{N,\Delta t}(t_{m+1}) = \Pi_{N} \left( \hat{\xi}^{\ABEM}_{N,\Delta t}(t_{m}) + f \left( \hat{\xi}^{\ABEM}_{N,\Delta t}(t_{m}) \right) \Delta t + g \left( \hat{\xi}^{\ABEM}_{N,\Delta t}(t_{m}) \right) \Delta B_{m} \right),
\end{equation}
where $\Delta B_{m} = B(t_{m+1}) - B(t_{m})$. Replacing $f$ and $g$ in~\eqref{eq:ABEMdef} with $f_{N}$ and $g_{N}$, respectively, would make it possible for scheme to get stuck at the boundary points of $\CDDN$. Also note that equation~\eqref{eq:ABEMdef} can equivalently be written as
\begin{equation}\label{eq:ABEMdef2}
  \begin{split}
    \hat{\xi}^{\ABEM}_{N,\Delta t}(t_{m+1}) = \min \{ b - 1/N, \max \{ &a + 1/N, \hat{\xi}^{\ABEM}_{N,\Delta t}(t_{m}) \\ &+ f(\hat{\xi}^{\ABEM}_{N,\Delta t}(t_{m})) \Delta t + g(\hat{\xi}^{\ABEM}_{N,\Delta t}(t_{m})) \Delta B_{m} \} \}.
  \end{split}
\end{equation}
We implement the ABEM scheme in Section~\ref{sec:numExp} using~\eqref{eq:ABEMdef2}. Note that if $b = \infty$, then the minimum in equation~\eqref{eq:ABEMdef2} will always the evaluated to be equal to the second term. In order words, the above definition of the ABEM scheme include both the one-finite-boundary case and the two-finite-boundaries case. See \cite{MR1357657,SLOMINSKI1994197} for details and convergence results for the standard projected EM scheme for RSDEs. Note that the convergence rate of $1/2-\epsilon$ (see Theorem~\ref{th:ABEMLpConv} below), for every $\epsilon>0$, for the ABEM scheme can (in the limit) be improved, but the $\epsilon$ cannot be removed (see \cite{MR1357657}). For our purposes, to confirm the convergence rate of the method of artificial barriers in Section~\ref{sec:artBarr}, convergence rate of $1/2-\epsilon$, for every $\epsilon>0$, suffices.

For the purpose of the convergence analysis, we extend the definition in equation~\eqref{eq:ABEMdef} to continuous time, and we introduce an auxiliary process $\hat{w}_{N,\Delta t}^{\ABEM}$:  For $t \in [t_{m},t_{m+1})$, let
\begin{equation}\label{eq:ABEMdefCont}
\hat{\xi}^{\ABEM}_{N,\Delta t}(t) = \Pi_{N} \left( \hat{\xi}^{\ABEM}_{N,\Delta t}(t_{m-1}) + f \left( \hat{\xi}^{\ABEM}_{N,\Delta t}(t_{m-1}) \right) \Delta t + g \left( \hat{\xi}^{\ABEM}_{N,\Delta t}(t_{m-1}) \right) \Delta B_{m-1} \right),
\end{equation}
and
\begin{equation*}
\hat{w}_{N,\Delta t}^{\ABEM}(t) = \hat{w}_{N,\Delta t}^{\ABEM}(t_{m-1}) + f \left( \hat{\xi}^{\ABEM}_{N,\Delta t}(t_{m-1}) \right) \Delta t + g \left( \hat{\xi}^{\ABEM}_{N,\Delta t}(t_{m-1}) \right) \Delta B_{m-1}.
\end{equation*}
Then (see Remark 1.4 in \cite{MR873889}) $(\xiNDt^{\ABEM},\hat{w}_{N,\Delta t}^{\ABEM},\xiNDt^{\ABEM}-\hat{w}_{N,\Delta t}^{\ABEM})$ is an associated triple on $\CDDN$  in the sense of Definition~\ref{def:RSDE}. In particular, $\hat{\xi}^{\ABEM}_{N,\Delta t} = \Gamma_{N}(\hat{w}_{N,\Delta t}^{\ABEM})$.

We introduce another auxiliary process
\begin{equation}\label{eq:deftwNDt}
\twNDt^{\ABEM}(t) = x_{0} + \int_{0}^{t} f(\xiNDt^{\ABEM}(s)) \diff s + \int_{0}^{t} g(\xiNDt^{\ABEM}(s)) \diff B(s),\ t \in [0,T],
\end{equation}
that coincides with $\hat{w}_{N,\Delta t}^{\ABEM}$ on the grid points $\{ t_{m}:\ m=0,\ldots,M \}$: $\twNDt^{\ABEM}(t_{m}) = \hat{w}_{N,\Delta t}^{\ABEM}(t_{m})$ for $m=0,\ldots,M$. Indeed by an induction argument over the index $m$: Suppose that $\twNDt^{\ABEM}(t_{m}) = \hat{w}_{N,\Delta t}^{\ABEM}(t_{m})$ for some $m=0,\ldots,M-1$. Then
\begin{align*}
\twNDt^{\ABEM}(t_{m+1}) &= \twNDt^{\ABEM}(t_{m}) + \int_{t_{m}}^{t_{m+1}} f(\xiNDt^{\ABEM}(s)) \diff s + \int_{t_{m}}^{t_{m+1}} g(\xiNDt^{\ABEM}(s)) \diff B(s) \\ &= \hat{w}_{N,\Delta t}^{\ABEM}(t_{m}) + f(\xiNDt^{\ABEM}(t_{m})) \Delta t + g(\xiNDt^{\ABEM}(t_{m})) \Delta B_{m} \\ &= \hat{w}_{N,\Delta t}^{\ABEM}(t_{m+1}),
\end{align*}
where we used that $\xiNDt^{\ABEM}$ is constant on $[t_{m},t_{m+1})$. Moreover, 
\begin{equation*}
(\txiNDt^{\ABEM},\twNDt^{\ABEM},\txiNDt^{\ABEM} - \twNDt^{\ABEM}),
\end{equation*}
with $\txiNDt^{\ABEM}(t) = \Gamma_{N}(\twNDt^{\ABEM})(t)$, for $t \in [0,T]$, is an associated triple on $\CDDN$ in the sense of Definition~\ref{def:RSDE}. 

It is clear from the definition of the ABEM scheme that it is boundary-preserving, which we state in the following proposition together with moment bounds.
\begin{proposition}\label{prop:ABEMmomBounds}
Suppose Assumptions~\ref{ass:Lip} and~\ref{ass:Feller} are satisfied and let $N \in \mathbb{N}$ and $\Delta t>0$. Then the ABEM scheme satisfies moment bounds
\begin{equation*}
\E \left[ \sup_{t \in [0,T]} \left| \hat{\xi}^{\ABEM}_{N,\Delta t}(t) \right|^{p} \right] \leq  
\begin{cases}
C(p,T,L_{f},L_{g},a,x_{0}), \text{ if } b = \infty,\\
\max(|a|^{p},|b|^{p}), \text{ if } |b| < \infty,
\end{cases}
\end{equation*}
for every $p \geq 2$, and only takes values in $\CDDN$
\begin{equation*}
\PP \left( \hat{\xi}^{\ABEM}_{N,\Delta t}(t) \in \CDDN,\ \forall t \in [0,T] \right)= 1.
\end{equation*}
The constant $C(p,T,L_{f},L_{g},a,x_{0})$ is independent of $N \in \mathbb{N}$ and $\Delta t$.
\end{proposition}
\begin{proof}[Proof of Proposition~\ref{prop:ABEMmomBounds}]
Let $p \geq 2$. The case of a bounded domain ($|a| + |b| < \infty$) is immediate since $|\xiNDt^{\ABEM}(t)| = |\Gamma_{N}(\tilde{w}_{N,\Delta t}^{\ABEM})|$ is bounded by $\max(|a|,|b|)$ in this case (see Lemma~\ref{lem:GNLipGB}). Suppose that $b = \infty$. We first reduce the desired moment bounds to the corresponding estimate for $\hat{w}_{N,\Delta t}^{\ABEM}$
\begin{equation*}
\E \left[ \sup_{t \in [0,T]} \left| \hat{\xi}^{\ABEM}_{N,\Delta t}(t) \right|^{p} \right] = \E \left[ \sup_{t \in [0,T]} \left| \Gamma_{N}\left(\hat{w}_{N,\Delta t}^{\ABEM}\right)(t) \right|^{p} \right] \leq C(p) \left( \E \left[ \sup_{t \in [0,T]} \left| \hat{w}_{N,\Delta t}^{\ABEM}(t) \right|^{p} \right] + |a|^{p} + 1 \right)
\end{equation*} 
using Lemma~\ref{lem:GNLipGB}. With the goal of applying Grönwall's lemma, we use the auxiliary process $\twNDt^{\ABEM}$ to bound the moments of $\hat{w}_{N,\Delta t}^{\ABEM}$
\begin{align*}
\sup_{t \in [0,r]} \left| \hat{w}_{N,\Delta t}^{\ABEM}(t) \right|^{p} = \sup_{m=0,\ldots,\ell_{\Delta t}(r)/\Delta t} \left| \hat{w}_{N,\Delta t}^{\ABEM}(t_{m}) \right|^{p} &= \sup_{m=0,\ldots,\ell_{\Delta t}(r)/\Delta t} \left| \tilde{w}^{\ABEM}_{N,\Delta t}(t_{m}) \right|^{p} \\ &\leq \sup_{t \in [0,r]} \left| \tilde{w}^{\ABEM}_{N,\Delta t}(t) \right|^{p},
\end{align*}
where we recall that $\ell_{\Delta t}(s) = t_{m}$ for $s \in [t_{m},t_{m+1})$ and where we also used that $\hat{w}_{N,\Delta t}^{\ABEM}$ and $\tilde{w}^{\ABEM}_{N,\Delta t}$ agree on the time grid points (see paragraph before this proposition). By~\eqref{eq:deftwNDt}, we can bound 
\begin{multline*}
\sup_{t \in [0,r]} \left| \twNDt^{\ABEM}(t) \right|^{p} \leq C(p) \left( |x_{0}|^{p} + \sup_{t \in [0,r]} \left| \int_{0}^{t} f(\Gamma_{N}(\hat{w}_{N,\Delta t}^{\ABEM})(s)) \diff s \right|^{p} \right. \\ \left. + \sup_{t \in [0,r]} \left| \int_{0}^{t} g(\Gamma_{N}(\hat{w}_{N,\Delta t}^{\ABEM})(s)) \diff B(s) \right|^{p} \right).
\end{multline*}
We use Jensen's inequality for integrals and linear growth of $f$ on $D$ (see Assumption~\ref{ass:Lip}) and of $\Gamma_{N}$ (see~\eqref{eq:GNGB}) to bound
\begin{equation*}
\sup_{t \in [0,r]} \left| \int_{0}^{t} f(\Gamma_{N}(\hat{w}_{N,\Delta t}^{\ABEM})(s)) \diff s \right|^{p} \leq C(p,T,L_{f}) \left( \int_{0}^{r} \sup_{t \in [0,s]} \left| \hat{w}_{N,\Delta t}^{\ABEM}(t) \right|^{p} \diff s + |a|^{p} + 1 \right).
\end{equation*}
For the Itô integral, we use the BDG inequality, Jensen's inequality for integrals, and linear growth of $g$ on $D$ (see Assumption~\ref{ass:Lip}) and of $\Gamma_{N}$ (see~\eqref{eq:GNGB}) to estimate
\begin{equation*}
\E \left[ \sup_{t \in [0,r]} \left| \int_{0}^{t} g(\Gamma_{N}(\hat{w}_{N,\Delta t}^{\ABEM})(s)) \diff B(s) \right|^{p} \right] \leq C(p,T,L_{g}) \left( \int_{0}^{r} \E \left[ \sup_{t \in [0,s]} \left| \hat{w}_{N,\Delta t}^{\ABEM}(t) \right|^{p} \right] \diff s + |a|^{p} + 1 \right)
\end{equation*}
Thus, together we have that
\begin{align*}
\E \left[ \sup_{t \in [0,r]} \left| \hat{w}_{N,\Delta t}^{\ABEM}(t) \right|^{p} \right] &\leq \E \left[ \sup_{t \in [0,r]} \left| \twNDt^{\ABEM}(t) \right|^{p} \right] \\ &\leq C(p,T,L_{f},L_{g}) \left( |x_{0}|^{p} + \int_{0}^{r} \E \left[ \sup_{t \in [0,s]} \left| \hat{w}_{N,\Delta t}^{\ABEM}(t) \right|^{p} \right] \diff s + |a|^{p} + 1 \right)
\end{align*}
An application of Grönwall's lemma now gives the desired estimate.
\end{proof}

The following theorem establishes $L^{p}(\Omega)$-convergence, for every $p \geq 2$, of the approximation $\hat{\xi}^{\ABEM}_{N,\Delta t}$ to the solution $X$ of the considered SDE~\eqref{eq:SDE}.
\begin{theorem}\label{th:ABEMLpConv}
Suppose Assumptions~\ref{ass:Lip} and~\ref{ass:Feller} are satisfied and let $N \in \mathbb{N}$ and $\Delta t>0$. Then there exists a constant $C(p,T,L_{f},L_{g},a,b,x_{0})>0$ such that
\begin{equation*}
\E \left[ \sup_{t \in [0,T]} \left| \hat{\xi}^{\ABEM}_{N,\Delta t}(t) - X(t) \right|^{p} \right] \leq C(p,T,L_{f},L_{g},a,b,x_{0}) \left( N^{-p} + \Delta t^{p/2-\epsilon} \right),
\end{equation*}
for every $p \geq 2$ and for every $\epsilon>0$, where the constant $C(p,T,L_{f},L_{g},a,b,x_{0})$ is independent of $N \in \mathbb{N}$ and $\Delta t$. If $b=\infty$, then the constant is independent of $b$.
\end{theorem}
\begin{proof}[Proof of Theorem~\ref{th:ABEMLpConv}]
We verify that Assumption~\ref{ass:BPRSDEscheme} is fulfilled and the result then follows from Theorem~\ref{th:BPschemeConv}. The proof follows the same proof idea as in \cite{MR1357657} with the extra issue that the coefficient functions of the RSDE for $\xi_{N}$ in~\eqref{eq:RmodSDE} are not the same as in the definition of the Artificial Barriers EM scheme $\hat{\xi}^{\ABEM}_{N,\Delta t}$. In view of Proposition~\ref{prop:ABEMmomBounds}, it remains to prove that
\begin{equation*}
\E \left[ \sup_{t \in [0,T]} |\hat{\xi}^{\ABEM}_{N,\Delta t}(t) - \xi_{N}(t)|^{p} \right] \leq C(p,T,L_{f},L_{g},a,b,x_{0}) \left( N^{-p} + \Delta t^{p/2-\epsilon} \right),
\end{equation*}
with a constant independent of $N \in \mathbb{N}$ and of $\Delta t>0$ and independent of $b$ if $b=\infty$. To this end, we split the error as
\begin{equation}\label{eq:ABEMproofErrSplit}
\left| \hat{\xi}^{\ABEM}_{N,\Delta t}(t) - \xi_{N}(t) \right|^{p} \leq C(p) \left( \left| \hat{\xi}^{\ABEM}_{N,\Delta t}(t) - \txiNDt^{\ABEM}(t) \right|^{p} + \left| \txiNDt^{\ABEM}(t) - \xi_{N}(t) \right|^{p} \right)
\end{equation}
and estimate each term separately.

We first consider the first term on the right hand side of equation~\eqref{eq:ABEMproofErrSplit}. By the Lipschitz property of the Skorokhod map $\Gamma_{N}$ (see Lemma~\ref{lem:GNLipGB}), we can bound
\begin{align*}
\sup_{t \in [0,r]} \left| \hat{\xi}^{\ABEM}_{N,\Delta t}(t) - \txiNDt^{\ABEM}(t) \right|^{p} &= \sup_{t \in [0,r]} \left| \Gamma_{N}(\hat{w}_{N,\Delta t}^{\ABEM})(t) - \Gamma_{N}(\twNDt^{\ABEM})(t) \right|^{p} \\ &\leq C(p) \sup_{t \in [0,r]} \left| \hat{w}_{N,\Delta t}^{\ABEM}(t) - \twNDt^{\ABEM}(t) \right|^{p}.
\end{align*}
We use the definition of $\twNDt^{\ABEM}(t)$, that $\hat{w}_{N,\Delta t}^{\ABEM}(t)$ and $\twNDt^{\ABEM}(t)$ coincide on the grid points and that $\hat{w}_{N,\Delta t}^{\ABEM}(t)$ is piecewise constant to see that
\begin{align*}
\twNDt^{\ABEM}(t) &= \hat{w}_{N,\Delta t}^{\ABEM}(t_{m}) + \int_{t_{m}}^{t} f(\xiNDt^{\ABEM}(s)) \diff s + \int_{t_{m}}^{t} g(\xiNDt^{\ABEM}(s)) \diff B(s) \\ &= \hat{w}_{N,\Delta t}^{\ABEM}(t) + \int_{t_{m}}^{t} f(\xiNDt^{\ABEM}(s)) \diff s + \int_{t_{m}}^{t} g(\xiNDt^{\ABEM}(s)) \diff B(s) \\ &= \hat{w}_{N,\Delta t}^{\ABEM}(t) + f(\xiNDt^{\ABEM}(t_{m}))(t-t_{m}) + g(\xiNDt^{\ABEM}(t_{m}))(B(t)-B(t_{m}))
\end{align*}
for $t \in [t_{m},t_{m+1})$. Moving $\hat{w}_{N,\Delta t}^{\ABEM}(t)$ to the left hand side we thus have
\begin{align*}
\E \left[ \sup_{t \in [0,r]} \left| \hat{w}_{N,\Delta t}^{\ABEM}(t) - \twNDt^{\ABEM}(t) \right|^{p} \right] &\leq C(p) \E \left[ \sup_{m=0,\ldots,M-1} \sup_{t \in [t_{m},t_{m+1})} \left| f(\xiNDt^{\ABEM}(t_{m}))(t-t_{m}) \right|^{p} \right] \\ &+ C(p) \E \left[ \sup_{m=0,\ldots,M-1} \sup_{t \in [t_{m},t_{m+1})} \left| g(\xiNDt^{\ABEM}(t_{m}))(B(t)-B(t_{m})) \right|^{p} \right]
\end{align*}
The first term on the right hand side can be bounded by $C(p,T,L_{f},L_{g},a,x_{0}) \Delta t^{p}$ if $b=\infty$ and by $C(p,L_{f},a,b) \Delta t^{p}$ if $|b| < \infty$ by linear growth of $f$ on $D$ in~\eqref{eq:fLG} and the moment bounds on $\xiNDt^{\ABEM}$ in Proposition~\ref{prop:ABEMmomBounds}. For the second term, we first use Cauchy-Schwarz inequality to separate the term involving $g$ and the term involving $B(t)-B(t_{m})$ as such:
\begin{align*}
\E \left[ \sup_{m=0,\ldots,M-1} \sup_{t \in [t_{m},t_{m+1})} \left| g(\xiNDt^{\ABEM}(t_{m-1}))(B(t)-B(t_{m})) \right|^{p} \right] &\leq \left( \E \left[ \sup_{m=0,\ldots,M-1} \sup_{t \in [t_{m},t_{m+1})} \left| g(\xiNDt^{\ABEM}(t_{m})) \right|^{2p} \right] \right)^{1/2} \\ &\times \left(\E \left[ \sup_{m=0,\ldots,M-1} \sup_{t \in [t_{m},t_{m+1})} \left| (B(t)-B(t_{m})) \right|^{2p} \right] \right)^{1/2}
\end{align*}
of which the first term can be bounded by $C(p,T,L_{f},L_{g},a,x_{0}) \Delta t^{p}$ if $b=\infty$ and by $C(p,L_{g},a,b) \Delta t^{p}$ if $|b| < \infty$ by linear growth of $g$ on $D$ in~\eqref{eq:gLG} and the moment bounds on $\xiNDt^{\ABEM}$ in Proposition~\ref{prop:ABEMmomBounds}. The second term can be bounded by
\begin{equation}\label{eq:boundBM}
\left(\E \left[ \sup_{m=0,\ldots,M-1} \sup_{t \in [t_{m},t_{m+1})} \left| (B(t)-B(t_{m})) \right|^{2p} \right] \right)^{1/2} \leq C(p,T) \Delta t^{p/2-\epsilon},
\end{equation}
for every $\epsilon>0$ (see, for example, Lemma $3$ in \cite{SLOMINSKI1994197}). Note that the bound in~\eqref{eq:boundBM} can be improved (in the limit) here by using, for example, Lemma $4.4$ in \cite{MR1357657} ($p=2$) or Lemma A.$4$ in \cite{MR1840835}. Thus, we have the estimate
\begin{equation}\label{eq:ABEMproof1st}
\E \left[ \sup_{t \in [0,r]} \left| \xiNDt^{\ABEM}(t_{m}) - \txiNDt^{\ABEM}(t) \right|^{p} \right] \leq C(p,T,L_{f},L_{g},a,b,x_{0}) \Delta t^{p/2 - \epsilon},
\end{equation}
where the constant is independent of $b$ and $x_{0}$ if $|b|<\infty$, for the first term on the right hand side of equation~\eqref{eq:ABEMproofErrSplit}.

We now consider the second term of the right hand side of equation~\eqref{eq:ABEMproofErrSplit}. We first, again, use the Lipschitz estimate for the Skorokhod map $\Gamma_{N}$ in Lemma~\ref{lem:GNLipGB}
\begin{align*}
\E \left[ \sup_{t \in [0,r]} \left| \txiNDt^{\ABEM}(t) - \xi_{N}(t) \right|^{p} \right] &= \E \left[ \sup_{t \in [0,r]} \left| \Gamma_{N}(\twNDt^{\ABEM})(t) - \Gamma_{N}(\wN)(t) \right|^{p} \right] \\ &\leq C(p) \E \left[ \sup_{t \in [0,r]} \left| \twNDt^{\ABEM}(t) - \wN(t) \right|^{p} \right].
\end{align*}
Recall that
\begin{equation*}
\twNDt^{\ABEM}(t) = x_{0} + \int_{0}^{t} f(\xiNDt^{\ABEM}(s)) \diff s + \int_{0}^{t} g(\xiNDt^{\ABEM}(s)) \diff B(s)
\end{equation*}
\begin{equation*}
\wN(t) = x_{0} + \int_{0}^{t} f_{N}(\xiN(s)) \diff s + \int_{0}^{t} g_{N}(\xiN(s)) \diff B(s)
\end{equation*}
where $f_{N}$ and $g_{N}$ are the modified drift and diffusion coefficient functions introduced in Section~\ref{sec:artBarr}. Therefore
\begin{align*}
\E \left[ \sup_{t \in [0,r]} \left| \twNDt^{\ABEM}(t) - \wN(t) \right|^{p} \right] &\leq C(p) \E \left[ \sup_{t \in [0,r]} \left| \int_{0}^{t} f(\xiNDt^{\ABEM}(s)) - f_{N}(\xiN(s)) \diff s \right|^{p} \right] \\ &+ C(p) \E \left[ \sup_{t \in [0,r]} \left| \int_{0}^{t} g(\xiNDt^{\ABEM}(s)) - g_{N}(\xiN(s)) \diff B(s) \right|^{p} \right] \\ &\leq C(p,T,L_{f},L_{g},a,b) \int_{0}^{t}  \E \left[ \sup_{t \in [0,s]} \left| \xiNDt^{\ABEM}(t) - \xiN(t) \right|^{p} \right] \diff s \\ &+ C(p,T,L_{f},L_{g},a,b) N^{-p} \left( 1 + \E \left[ \sup_{t \in [0,T]} \left| \xiNDt^{\ABEM}(t) \right|^{p} \right] \right),
\end{align*}
with the constants being independent of $a$ and $b$ if $b=\infty$, where we also used Lemma~\ref{lem:IntffNerr} and Lemma~\ref{lem:IntggNerr} (and connecting remarks). Inserting this, the moment bounds in Proposition~\ref{prop:ABEMmomBounds} and the estimate~\eqref{eq:ABEMproof1st} into the splitting of the error in~\eqref{eq:ABEMproofErrSplit}, we obtain
\begin{align*}
\E \left[ \sup_{t \in [0,r]} \left| \hat{\xi}_{N,\Delta t}^{\ABEM}(t) - \xi_{N}(t) \right|^{p} \right] &\leq C(p,T,L_{f},L_{g},a,b,x_{0}) \left( N^{-p} + \Delta t^{p/2-\epsilon} \right) \\ &+ C(p,T,L_{f},L_{g},a,b) \int_{0}^{t} \E \left[ \sup_{t \in [0,s]} \left| \hat{\xi}_{N,\Delta t}^{\ABEM}(t) - \xi_{N}(t) \right|^{p} \right] \diff s
\end{align*}
and by applying Grönwall's inequality we deduce that
\begin{equation*}
\E \left[ \sup_{t \in [0,T]} \left| \hat{\xi}_{N,\Delta t}^{\ABEM}(t) - \xi_{N}(t) \right|^{p} \right] \leq C(p,T,L_{f},L_{g},a,b,x_{0}) (N^{-p} + \Delta t^{p/2-\epsilon}),
\end{equation*}
where the constant is independent of $b$ and $x_{0}$ if $b=\infty$, which is the desired estimate and this concludes the proof.
\end{proof}

As in Corollary~\ref{corr:BPschemeConv}, we couple $N \in \mathbb{N}$ and $\Delta t>0$ and obtain almost sure pathwise convergence with rate $1/2-\epsilon$ for every $\epsilon>0$. We verified in the proof of Theorem~\ref{th:ABEMLpConv} that $\hat{\xi}^{\ABEM}_{N,\Delta t}$ satisfy Assumption~\ref{ass:BPRSDEscheme}, hence Corollary~\ref{corr:ABEMasConv} follows from Corollary~\ref{corr:BPschemeConv}.
\begin{corollary}\label{corr:ABEMasConv}
Suppose Assumptions~\ref{ass:Lip} and~\ref{ass:Feller} are satisfied and let $N \in \mathbb{N}$ and $\Delta t = N^{-2}$. Then there exists, for every $\epsilon>0$, a random variable $\eta_{\epsilon}$, with $\E \left[ \eta_{\epsilon}^{p} \right] < \infty$ for every $p \geq 1$, such that
\begin{equation*}
\sup_{t \in [0,T]} |\hat{\xi}^{\ABEM}_{N,\Delta t}(t) - X(t)| \leq \eta_{\epsilon} \Delta t^{1/2 - \epsilon},
\end{equation*}
almost surely. Moreover, $\eta_{\epsilon}$ is independent of $N \in \mathbb{N}$ and $\Delta t$.
\end{corollary}

\section{Artificial Barriers Euler--Peano}\label{sec:ABEP}
In this section, we combine the method of artificial barriers developed in Section~\ref{sec:artBarr} with a modification of the classical Euler--Peano (EP) scheme for RSDEs applied to the RSDE with modified coefficients in~\eqref{eq:RmodSDE}. We refer to the resulting scheme as the Artificial Barriers Euler--Peano scheme and we abbreviate it as ABEP. We first provide the definition of the scheme and then we verify that the ABEP scheme satisfy Assumption~\ref{ass:BPRSDEscheme} in Section~\ref{sec:artBarr}. By Theorem~\ref{th:BPschemeConv}, we can then conclude that the ABEP scheme is boundary-preserving and converges to the solution $X$ of the considered SDE~\eqref{eq:SDE}. 

The classical EP scheme for RSDEs is known to exhibit $L^{p}(\Omega)$-convergence of order $1/2$ (see, for example, \cite{SLOMINSKI1994197}). The goal in this section is to show that this is also true for the ABEP scheme for the RSDE in~\eqref{eq:RmodSDE} with an $L^{p}(\Omega)$-error constant that can be taken to be independent of $N$, and hence that the ABEP scheme satisfies Assumption~\ref{ass:BPRSDEscheme}. 

We would like to make an important remark before proceeding to the definition of the ABEP scheme. There are explicit expressions for the EP and ABEP schemes (see equation~\eqref{eq:ABEPexplicit} for the ABEP scheme). However, as the supremum and infimum in Definition~\ref{def:Gexplicit} and Definition~\ref{def:G0aExplicit} are over infinitely many time points, the EP and ABEP schemes are not known to be implementable.  In other words, as of now, the ABEP scheme is of more theoretical than practical use. We include the ABEP scheme as it may be used to construct other implementable schemes (for example, a modified version of the Euler scheme for RSDEs studied in \cite{MR1341164}).

We now provide a step-by-step construction of the ABEP scheme. We refer the reader not familiar with RSDEs to equation~\eqref{eq:ABEPexplicit} for a representation of the ABEP scheme without any mentioning of RSDEs. Recall that the classical EM scheme for the SDE in~\eqref{eq:SDEmain} is given by
\begin{equation*}
X^{\EM}_{\Delta t}(t_{m+1}) = X^{\EM}_{\Delta t}(t_{m}) + f(X^{\EM}_{\Delta t}(t_{m})) \Delta t + g(X^{\EM}_{\Delta t}(t_{m})) \Delta B_{m},
\end{equation*}
for $m=0,\ldots,M-1$, and, for the purpose of defining the ABEP scheme, we extend the above to continuous time by requiring that $X^{\EM}_{\Delta t}(t)$ solves the following SDE with piecewise constant coefficients
\begin{equation*}
\diff X^{\EM}_{\Delta t}(t) = f(X^{\EM}_{\Delta t}(\ell_{\Delta t}(t))) \diff t + g(X^{\EM}_{\Delta t}(\ell_{\Delta t}(t))) \diff B(t),\ t \in [0,T],
\end{equation*}
where we recall that $\ell_{\Delta t}(t) = t_{m}$ for $t \in [t_{m},t_{m+1})$, with initial value $X^{\EM}_{\Delta t}(0) = x_{0}$. We now add a lower barrier process $\hat{L}_{N,\Delta t}^{\ABEP}$ and an upper barrier process $\hat{U}_{N,\Delta t}^{\ABEP}$, in the sense of Section~\ref{sec:RSDE}, to force $X^{\EM}_{\Delta t}$ to stay in $\CDDN$:
\begin{equation*}
\diff \hat{\xi}^{\ABEP}_{N,\Delta t}(t) = f(\hat{\xi}^{\ABEP}_{N,\Delta t}(\ell_{\Delta t}(t))) \diff t + g(\hat{\xi}^{\ABEP}_{N,\Delta t}(\ell_{\Delta t}(t))) \diff B(t) - \diff \hat{U}_{N,\Delta t}^{\ABEP}(t) + \diff \hat{L}_{N,\Delta t}^{\ABEP}(t),\ t \in [0,T],
\end{equation*}
or equivalently in integral from
\begin{equation}\label{eq:ABEPdef}
\hat{\xi}^{\ABEP}_{N,\Delta t}(t) = x_{0} + \int_{0}^{t} f(\hat{\xi}^{\ABEP}_{N,\Delta t}(\ell_{\Delta t}(s))) \diff s + \int_{0}^{t} g(\hat{\xi}^{\ABEP}_{N,\Delta t}(\ell_{\Delta t}(s))) \diff B(s) - \hat{U}_{N,\Delta t}^{\ABEP}(t) + \hat{L}_{N,\Delta t}^{\ABEP}(t),\ t \in [0,T].
\end{equation}
Observe that, if $b=\infty$ then $\hat{U}_{N,\Delta t}^{\ABEP}$ can be removed with no change. We also introduce the process
\begin{equation}\label{eq:wABEPdef}
\hat{w}_{N,\Delta t}^{\ABEP}(t) = x_{0} + \int_{0}^{t} f(\hat{\xi}^{\ABEP}_{N,\Delta t}(\ell_{\Delta t}(s))) \diff s + \int_{0}^{t} g(\hat{\xi}^{\ABEP}_{N,\Delta t}(\ell_{\Delta t}(s))) \diff B(s),\ t \in [0,T],
\end{equation}
to obtain an associated triple $(\hat{\xi}^{\ABEP}_{N,\Delta t},\hat{w}_{N,\Delta t}^{\ABEP},\hat{L}_{N,\Delta t}^{\ABEP}-\hat{U}_{N,\Delta t}^{\ABEP})$ in the sense of Section~\ref{sec:RSDE}. Then, in particular, we have that $\hat{\xi}^{\ABEP}_{N,\Delta t} = \Gamma_{N}(\hat{w}_{N,\Delta t}^{\ABEP})$. Note that~\eqref{eq:wABEPdef} on $[\ell_{\Delta t}(s),\ell_{\Delta t}(s)+\Delta t]$ is the classical EM scheme starting from the value $\hat{\xi}^{\ABEP}_{N,\Delta t}(\ell_{\Delta t}(s))$. Moreover, using the explicit expressions for the Skorokhod map in Definition~\ref{def:G0aExplicit}, we have the following representation for $\hat{\xi}^{\ABEP}_{N,\Delta t}$ in the two-finite-boundaries case ($|b|<\infty$)
\begin{equation}\label{eq:ABEPexplicit}
  \begin{split}
    \hat{\xi}^{\ABEP}_{N,\Delta t}(t) &= a + \frac{1}{N} + \hat{w}_{N,\Delta t}^{\ABEP}(t) + \sup_{s \in [0,t]} \max(0,-\hat{w}_{N,\Delta t}^{\ABEP}(s)) \\ &- \sup_{s \in [0,t]} \left( \min \left( \max \left(0, \hat{w}_{N,\Delta t}^{\ABEP}(s) + \sup_{r \in [0,s]} \max \left( 0,-\hat{w}_{N,\Delta t}^{\ABEP}(r) \right) - \left( b - a - \frac{2}{N},\right) \right), \right. \right. \\ &\left. \left. \inf_{u \in [s,t]} \left( \hat{w}_{N,\Delta t}^{\ABEP}(u) + \sup_{v \in [0,u]} \max \left( 0 , -\hat{w}_{N,\Delta t}^{\ABEP}(v) \right) \right) \right) \right),\ t \in [0,T].
  \end{split}
\end{equation}
The involved formula in~\eqref{eq:ABEPexplicit} is simplified in the one-finite-boundary case ($b=\infty$), but the ABEP scheme is not implementable in this case either.

The following proposition provides moment bounds and boundary-preservation of the ABEP scheme. Since the ABEP scheme is defined in~\eqref{eq:ABEPdef} as the solution to an RSDE, the boundary-preserving property of the ABEP scheme follows directly. 
\begin{theorem}\label{th:ABEPmomBounds}
Suppose Assumptions~\ref{ass:Lip} and~\ref{ass:Feller} are satisfied and let $N \in \mathbb{N}$ and $\Delta t>0$. Then the ABEP scheme satisfies moment bounds
\begin{equation*}
\E \left[ \sup_{t \in [0,T]} \left| \hat{\xi}^{\ABEP}_{N,\Delta t}(t)\right|^{p} \right] \leq
\begin{cases}
C(p,T,L_{f},L_{g},a,x_{0}), \text{ if } b = \infty,\\
\max(|a|^{p},|b|^{p}), \text{ if } |b| < \infty,
\end{cases}
\end{equation*}
for every $p \geq 2$, and only takes values in $\CDDN$
\begin{equation*}
\PP \left( \hat{\xi}^{\ABEP}_{N,\Delta t}(t) \in \CDDN,\ \forall t \in [0,T] \right)= 1.
\end{equation*}
The constant $C(p,T,L_{f},L_{g},a,x_{0})$ is independent of $N \in \mathbb{N}$ and $\Delta t$.
\end{theorem}
\begin{proof}[Proof of Theorem~\ref{th:ABEPmomBounds}]
Let $p \geq 2$. The case of a bounded domain ($|a| + |b| < \infty$) is immediate since $|\xiNDt^{\ABEP}(t)| = |\Gamma_{N}(\hat{w}_{N,\Delta t}^{\ABEP})|$ is bounded by $\max(|a|,|b|)$ in this case (see Lemma~\ref{lem:GNLipGB}). Suppose now that $b=\infty$. We reduce the desired moment estimates on $\xiNDt^{\ABEP}$ to the corresponding moment estimates for $\hat{w}_{N,\Delta t}^{\ABEP}$ using Lemma~\ref{lem:GNLipGB}
\begin{equation*}
\E \left[ \sup_{t \in [0,T]} \left| \hat{\xi}_{N,\Delta t}^{\ABEP}(t)\right|^{p} \right] = \E \left[ \sup_{t \in [0,T]} \left| \Gamma_{N} \left(\hat{w}_{N,\Delta t}^{\ABEP}\right)(t) \right|^{p} \right] \leq C(p) \left( \E \left[ \sup_{t \in [0,T]} \left| \hat{w}_{N,\Delta t}^{\ABEP}(t) \right|^{p} \right] + |a|^{p} + 1 \right).
\end{equation*}

We use~\eqref{eq:wABEPdef} to estimate
\begin{align*}
&\sup_{t \in [0,r]} |\hat{w}_{N,\Delta t}^{\ABEP}(t)|^{p} \\ &\leq C(p) \left( |x_{0}|^{p} + \sup_{t \in [0,r]} \left| \int_{0}^{t} f(\Gamma_{N}(\hat{w}_{N,\Delta t}^{\ABEP})(\ell_{\Delta t}(s))) \diff s \right|^{p} + \sup_{t \in [0,r]} \left| \int_{0}^{t} g(\Gamma_{N}(\hat{w}_{N,\Delta t}^{\ABEP})(\ell_{\Delta t}(s))) \diff B(s) \right|^{p} \right) \\ &= C(p) \left( |x_{0}|^{p} + I_{1} + I_{2} \right).
\end{align*}
For $I_{1}$, we use Jensen's inequality for integrals and the linear growth of $f$ on $D$ (see Assumption~\ref{ass:Lip}) and of $\Gamma_{N}$ (see Lemma~\ref{lem:GNLipGB}) to obtain
\begin{equation*}
I_{1} \leq C(p,T,L_{f}) \left( \int_{0}^{r} \sup_{t \in [0,s]} |\hat{w}_{N,\Delta t}^{\ABEP}(t)|^{p} \diff s + |a|^{p} + 1 \right)
\end{equation*}
and for $I_{2}$ we use the BDG inequality, Jensen's inequality for integrations and linear growth of $g$ on $D$ (see Assumption~\ref{ass:Lip}) and of $\Gamma_{N}$ (see~\eqref{eq:GNGB}) to estimate
\begin{equation*}
\E \left[ I_{2} \right] \leq C(p,T,L_{g}) \left( \int_{0}^{r} \E \left[ \sup_{t \in [0,s]} |\hat{w}_{N,\Delta t}^{\ABEP}(t)|^{p} \right] \diff s + |a|^{p} + 1 \right).
\end{equation*}
Thus, in total
\begin{equation*}
\E \left[ \sup_{t \in [0,r]} |\hat{w}_{N,\Delta t}^{\ABEP}(t)|^{p} \right] \leq C(p,T,L_{f},L_{g}) \left( |x_{0}|^{p} + \int_{0}^{r} \E \left[ \sup_{t \in [0,s]} |\hat{w}_{N,\Delta t}^{\ABEP}(t)|^{p} \right] \diff s + |a|^{p} + 1 \right)
\end{equation*}
and Grönwall's lemma gives the desired estimate
\begin{equation*}
\E \left[ \sup_{t \in [0,T]} |\hat{w}_{N,\Delta t}^{\ABEP}(t)|^{p} \right] \leq C(p,T,L_{f},L_{g},a,x_{0}).
\end{equation*}
\end{proof}

We now state and prove the convergence theorem of this section; namely, the $L^{p}(\Omega)$-convergence, for every $p \geq 2$, of $\hat{\xi}^{\ABEP}_{N,\Delta t}$ to the solution $X$ of the considered SDE~\eqref{eq:SDE}. As for Theorem~\ref{th:ABEMLpConv}, the proof of Theorem~\ref{th:ABEPLpConv} verifies that Assumption~\ref{ass:BPRSDEscheme} is fulfilled and the result then follows from Theorem~\ref{th:BPschemeConv}. Note that convergence of the EP scheme is known (see, for example, \cite{SLOMINSKI1994197}), we include it since, in contrast to the EP scheme, the coefficient functions of the RSDE for $\xi_{N}$ are not the same as in the definition of the modified EP scheme $\hat{\xi}^{\ABEP}_{N,\Delta t}$.
\begin{theorem}\label{th:ABEPLpConv}
Suppose Assumptions~\ref{ass:Lip} and~\ref{ass:Feller} are satisfied and let $N \in \mathbb{N}$ and $\Delta t>0$. Then
\begin{equation*}
\E \left[ \sup_{t \in [0,T]} \left| \hat{\xi}^{\ABEP}_{N,\Delta t}(t) - X(t) \right|^{p} \right] \leq C(p,T,L_{f},L_{g},a,b,x_{0}) \left( N^{-p} + \Delta t^{p/2} \right),
\end{equation*}
for every $p \geq 2$, where the constant is independent of $N \in \mathbb{N}$ and $\Delta t$. If $b=\infty$, then the constant is independent of $b$.
\end{theorem}
\begin{proof}[Proof of Theorem~\ref{th:ABEPLpConv}]
In view of Theorem~\ref{th:BPschemeConv} and Theorem~\ref{th:ABEPmomBounds}, it remains to prove that
\begin{equation}\label{eq:ABEMerrEst}
\E \left[ \sup_{t \in [0,T]} |\hat{\xi}^{\ABEP}_{N,\Delta t}(t) - \xi_{N}(t)|^{p} \right] \leq C(p,T,L_{f},L_{g},a,b,x_{0}) \left( N^{-p} + \Delta t^{p/2} \right),
\end{equation}
for some constant independent of $N \in \mathbb{N}$ and independent of $\Delta t>0$, and independent of $b$ if $b=\infty$.  

We first reduce the estimate to the corresponding estimate for $\hat{w}_{N,\Delta t}^{\ABEP}(t) - w_{N}(t)$
\begin{align*}
\E \left[ \sup_{t \in [0,T]} |\xiNDt^{\ABEP}(t) - \xi_{N}(t)|^{p} \right] &= \E \left[ \sup_{t \in [0,T]} |\Gamma_{N}(\hat{w}_{N,\Delta t}^{\ABEP})(t) - \Gamma_{N} (w_{N})(t) |^{p} \right] \\ &\leq C(p) \E \left[ \sup_{t \in [0,T]} |\hat{w}_{N,\Delta t}^{\ABEP}(t) - w_{N}(t)|^{p} \right],
\end{align*}
using the Lipschitz property of $\Gamma_{N}$ in~\eqref{eq:GNLip}, and we prove that the right hand side can be bounded by the right hand side in equation~\eqref{eq:ABEMerrEst} using a Grönwall argument.

By definition of $\hat{w}_{N,\Delta t}^{\ABEP}$ (see~\eqref{eq:wABEPdef}) and $w_{N}$ (see~\eqref{eq:wNdef}), we can estimate
\begin{equation}\label{eq:wNDtwNerr}
\begin{split}
\sup_{t \in [0,r]} &|\hat{w}_{N,\Delta t}^{\ABEP}(t) - w_{N}(t)|^{p} \\ &\leq C(p) \sup_{t \in [0,r]} \left| \int_{0}^{t} f (\Gamma_{N}( \hat{w}_{N,\Delta t}^{\ABEP})(\ell_{\Delta t}(s))) - f_{N}(\Gamma_{N}(w_{N})(s)) \diff s \right|^{p} \\ &+ C(p) \sup_{t \in [0,r]} \left| \int_{0}^{t} g (\Gamma_{N}(\hat{w}_{N,\Delta t}^{\ABEP})(\ell_{\Delta t}(s))) - g_{N}(\Gamma_{N}(w_{N})(s)) \diff B(s) \right|^{p} \\ &= C(p)(J_{1} + J_{2}).
\end{split}
\end{equation}
$\E \left[ J_{1} \right]$ and $\E \left[ J_{2} \right]$ are estimated in the same way, the difference is that we use Lemma~\ref{lem:IntffNerr} for $\E \left[ J_{1} \right]$ and we use Lemma~\ref{lem:IntggNerr} for $\E \left[ J_{2} \right]$ . We provide the argument for $\E \left[ J_{2} \right]$, as the term involving the noise is (usually) the difficult part. Using (remarks to) Lemma~\ref{lem:IntggNerr}, the Lipschitz estimate for $\Gamma_{N}$ in~\eqref{eq:GNLip}, and moment bounds for $\xiNDt^{\ABEP} = \Gamma_{N}(\hat{w}_{N,\Delta t}^{\ABEP})$ in Theorem~\ref{th:ABEPmomBounds}, we can estimate
\begin{align*}
\E \left[ J_{2} \right] &\leq C(p,T,L_{g},a,b) \int_{0}^{r} \E \left[ |\Gamma_{N}(\hat{w}_{N,\Delta t}^{\ABEP})(\ell_{\Delta t}(s)) - \Gamma_{N}(w_{N})(s)|^{p} \right] \diff s \\ &+ C(p,T,L_{g},a,b) N^{-p} \left(1 + \E \left[ \sup_{t \in [0,T]} |\Gamma_{N}(\hat{w}_{N,\Delta t}^{\ABEP})(t)|^{p} \right] \right) \\ &\leq C(p,T,L_{f},L_{g},a,b,x_{0}) \int_{0}^{r} \E \left[ \sup_{t \in [0,s]} | \Gamma_{N} (\hat{w}_{N,\Delta t}^{\ABEP})(t) - \Gamma_{N}(w_{N})(t)|^{p} \right] \diff s \\ &+ C(p,T,L_{f},L_{g},a,b,x_{0}) \left( \int_{0}^{r} \E \left[ | \Gamma_{N} (\wN)(\ell_{\Delta t}(s)) - \Gamma_{N}(w_{N})(s)|^{p} \right] \diff s + N^{-p} \right) \\ &\leq C(p,T,L_{f},L_{g},a,b,x_{0}) \left( \int_{0}^{r} \E \left[ \sup_{t \in [0,s]} | \hat{w}_{N,\Delta t}^{\ABEP}(t) - w_{N}(t)|^{p} \right] \diff s + N^{-p} + \Delta t^{p/2} \right),
\end{align*}
where the constants are independent of $b$ if $b=\infty$ and
where we also used the time regularity in~\eqref{eq:xNtimeReg} to bound 
\begin{equation*}
\E \left[ |\xi_{N}(\ell_{\Delta t}(s))-\xi_{N}(s)|^{p} \right] \leq C(p,L_{f},L_{g},a,b,x_{0}) \Delta t^{p/2}.
\end{equation*}
Similar argument gives the same bound for $\E \left[ J_{1} \right]$
\begin{equation*}
\E \left[ J_{1} \right] \leq C(p,T,L_{f},L_{g},a,b,x_{0}) \left( \int_{0}^{r} \E \left[ \sup_{t \in [0,s]} | \hat{w}_{N,\Delta t}^{\ABEP}(t) - w_{N}(t)|^{p} \right] \diff s + N^{-p} + \Delta t^{p/2} \right),
\end{equation*}
where the constant is independent of $b$ if $b=\infty$. Therefore, applying $\E$ to equation~\eqref{eq:wNDtwNerr} and using the above bounds for $\E \left[ J_{1} \right]$ and $\E \left[ J_{2} \right]$, we obtain
\begin{multline*}
\E \left[ \sup_{t \in [0,r]} |\hat{w}_{N,\Delta t}^{\ABEP}(t) - w_{N}(t)|^{p} \right] \\ \leq C(p,T,L_{f},L_{g},a,b,x_{0}) \left( \int_{0}^{r} \E \left[\sup_{t \in [0,s]}| \hat{w}_{N,\Delta t}^{\ABEP}(t) - w_{N}(t)|^{p} \right]  \diff s + N^{-p} + \Delta t^{p/2} \right)
\end{multline*}
and applying Grönwall's inequality we conclude that
\begin{equation*}
\E \left[ \sup_{t \in [0,T]} |\hat{w}_{N,\Delta t}^{\ABEP}(t) - w_{N}(t)|^{p} \right] \leq C(p,T,L_{f},L_{g},a,b,x_{0}) \left( N^{-p} + \Delta t^{p/2} \right),
\end{equation*}
where the constants are independent of $b$ if $b=\infty$. This finishes the proof.
\end{proof}

From Theorem~\ref{th:ABEPLpConv}, we also obtain almost sure pathwise convergence in Corollary~\ref{corr:ABEPasConv}.
\begin{corollary}\label{corr:ABEPasConv}
Suppose Assumptions~\ref{ass:Lip} and~\ref{ass:Feller} are satisfied and let $N \in \mathbb{N}$ and $\Delta t = N^{-2}$. Then there exists for every $\epsilon>0$ a random variable $\eta_{\epsilon}$, with $\E \left[ \eta_{\epsilon}^{p} \right] < \infty$ for every $p \geq 1$, such that
\begin{equation*}
\sup_{t \in [0,T]} |\hat{\xi}^{\ABEP}_{N,\Delta t}(t) - X(t)| \leq \eta_{\epsilon} \Delta t^{1/2 - \epsilon},
\end{equation*}
almost surely. Moreover, $\eta_{\epsilon}$ is independent of $N \in \mathbb{N}$ and $\Delta t$.
\end{corollary}

\section{Numerical experiments}\label{sec:numExp}
We present numerical experiments to numerically confirm the theoretical results obtained in Theorem~\ref{th:ABEMLpConv}. Recall from Section~\ref{sec:artBarr} that $M \in \mathbb{N}$ denotes the number of subintervals $[t_{m},t_{m+1}]$ each of size $\Delta t = T/M$. In this section we strive to verify that the ABEM scheme achieves $L^{p}(\Omega)$-convergence rate of $1/2-$ as stated in Theorem~\ref{th:ABEMLpConv}, and to this end we couple $N$ and $\Delta t$ as $N = \Delta t^{-d}$ for some $d>0$. Let $\Delta B_{m} = B(t_{m+1}) - B(t_{m})$ denote the Brownian motion increment over the interval $[t_{m},t_{m+1}]$. 

We provide numerical experiments that illustrate boundary-preservation and $L^{2}(\Omega)$-convergence of order $1/2-$ of the ABEM scheme as proved in Section~\ref{sec:ABEM}. Boundary-preservation of the proposed ABEM scheme, denote by $\hat{\xi}^{\ABEM}$ below, is compared with the following schemes that lack boundary-preservation (see the following sections):
\begin{itemize}
\item the Euler--Maruyama scheme (denoted EM below), see for instance \cite{MR1214374}
$$
X^{\EM}_{m+1}=X^{\EM}_m+ f(X^{\EM}_{m})\Delta t + g(X^{\EM}_m) \Delta B_{m},
$$
\item the semi-implicit Euler--Maruyama scheme (denoted SEM below), see for instance \cite{MR1214374}
$$
X^{\SEM}_{m+1}=X^{\SEM}_m+ f(X^{\SEM}_{m+1})\Delta t + g(X^{\SEM}_m) \Delta B_{m},
$$
\item the tamed Euler scheme (denoted TE below), see for instance \cite{MR2985171, MR3543890}
$$
X^{\TE}_{m+1}=X^{\TE}_m+ f^{M}(X^{\TE}_{m})\Delta t + g^{M}(X^{\TE}_m) \Delta B_{m},
$$
where
$$
f^{M}(x) = \frac{f(x)}{1 + M^{-1/2} |f(x)| + M^{-1/2} | g(x)|^{2}}
$$
$$
g^{M}(x) = \frac{g(x)}{1 + M^{-1/2} |f(x)| + M^{-1/2} | g(x)|^{2}}.
$$
\end{itemize}
We present boundary-preservation and lack thereof in tables displaying the number of sample paths out of $100$ that only contained values in the domain $D$ and we present in loglog plots the $L^{2}(\Omega)$-errors given by
\begin{equation}\label{eq:L2err}
\left( \E \left[ \sup_{m=0,\ldots,M} |\hat{\xi}^{\ABEM}_{m} - X^{ref}_{m}|^{2} \right] \right)^{1/2}
\end{equation}
where the reference solution $X^{ref}$ is computed using the ABEM scheme with $\Delta t^{ref} = 10^{-7}$ (except for the SDE for geometric Brownian motion in Section~\ref{subsec:GBM} where we use the closed-form solution as reference). We approximate the expected value in~\eqref{eq:L2err} using $300$ Monte Carlo samples. Numerical experiments have verified that $300$ Monte Carlo samples is sufficient for the Monte Carlo error to be small enough to observe the order of $L^{2}(\Omega)$-convergence when approximating the expected value in~\eqref{eq:L2err}.

\subsection{Geometric Brownian motion}\label{subsec:GBM}
We first provide numerical experiments for the SDE for geometric Brownian motion given by
\begin{equation}\label{eq:gBM}
\diff X(t) = -X(t) \diff t + \lambda X(t) \diff B(t),\ t \in (0,T],
\end{equation}
with initial value $X(0)=x_{0} \in (0,\infty)$;
that is, $f(x) = -x$ and $g(x) = \lambda x$ are linear in the considered SDE in equation~\eqref{eq:SDE}. We introduce a parameter $\lambda>0$ to better illustrate that the EM, SEM and TE schemes do not preserve the boundaries of the SDE~\eqref{eq:gBM}. In this case $D = (0,\infty)$ and the SDE in~\eqref{eq:gBM} is an example of a one-finite-boundary case. The solution $X(t)$ of~\eqref{eq:gBM} is called geometric Brownian motion and is given by
\begin{equation}\label{eq:gBmExplicit}
X(t) = x_{0} \exp \left( - \left( 1 + \frac{\lambda^{2}}{2} \right) t + \lambda B(t) \right),\ t \in [0,T].
\end{equation}
We use~\eqref{eq:gBmExplicit} as the reference solution for these numerical experiments. Geometric Brownian motion is one of the fundamental stochastic processes used in financial mathematics to model stock price fluctuations (for example within the famous Black-Scholes framework \cite{MR3363443,MR496534}).

We first provide in Figure~\ref{num:gBM_path_comp} sample paths of the EM, SEM, TE and ABEM schemes that illustrate that the EM, SEM and TE schemes are not boundary-preserving. All sample paths in Figure~\ref{num:gBM_path_comp} use the same Brownian motion sample path for the computations. Note that the sample paths of EM, SEM and TE schemes almost coincide in Figure~\ref{num:gBM_path_comp}.

\begin{figure}[h!]
\begin{center}
  \includegraphics[width=0.8\textwidth]{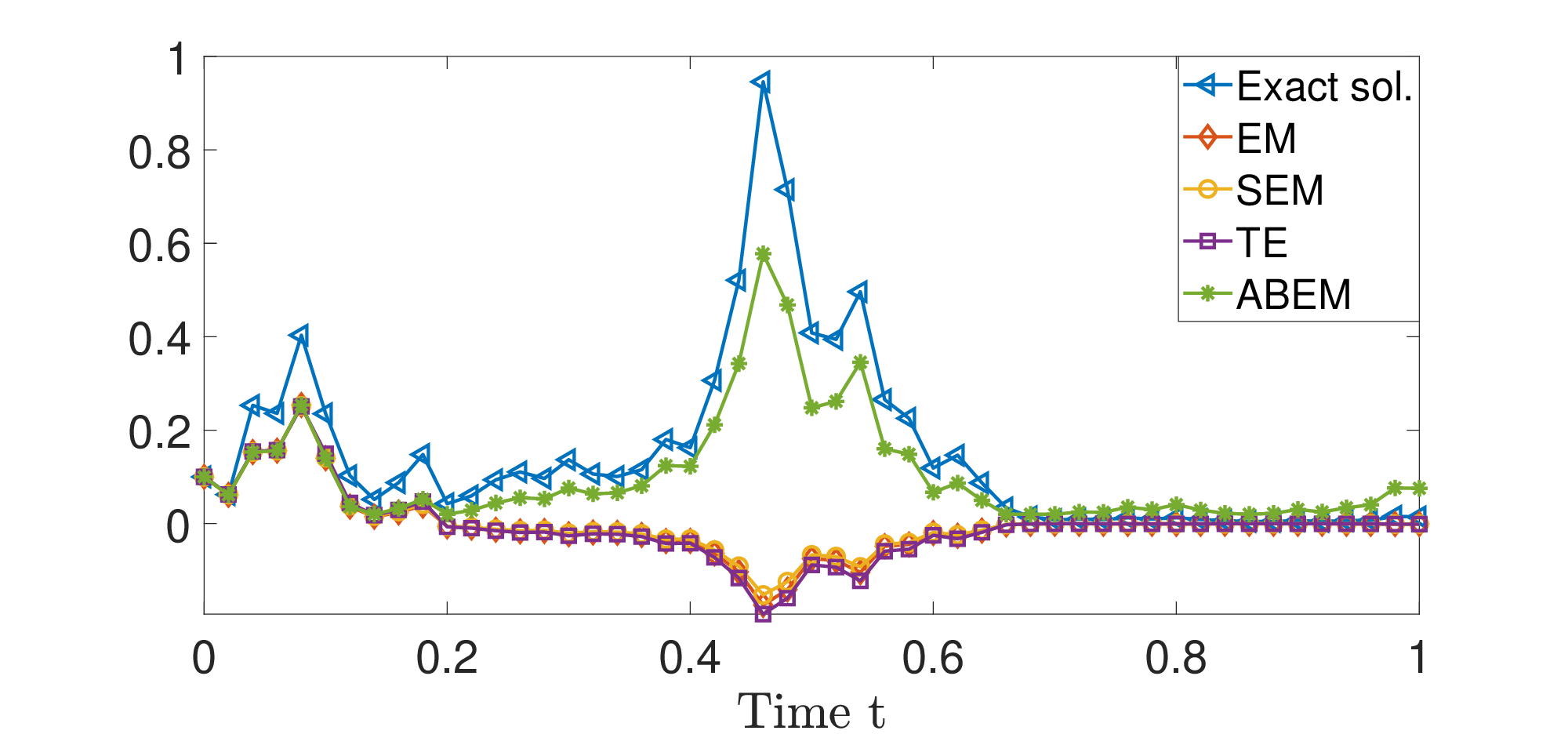}
  \caption{Path comparison of the Euler--Maruyama scheme (EM), the semi-implicit Euler--Maruyama scheme (SEM), the tamed Euler scheme (TE), and Artificial Barriers Euler--Maruyama scheme (ABEM) applied to the SDE for geometric Brownian motion in~\eqref{eq:gBM} with parameters $\lambda = 3$, $T = 1$, $x_{0}=0.1$, $d=1$ and $\Delta t = 0.02$.}\label{num:gBM_path_comp}
  \end{center}
\end{figure}

Next, we provide in Table~\ref{tb:gBM}, for some choices of $\lambda>0$, proportions out of $100$ samples of the EM, SEM, TE and ABEM schemes that only contain values in the domain $D = (0,\infty)$ of the SDE in~\eqref{eq:gBM}. Table~\ref{tb:gBM} confirms that the ABEM scheme is boundary-preserving and shows that the EM, SEM and TE schemes are not. As expected, the number of sample paths of the EM, SEM and TE schemes that only contain values in the domain $(0,\infty)$ decreases as $\lambda>0$ increases. We used the parameters $T=1$, $x_{0}=0.1$, $d=1/2$ and $\Delta t = 10^{-2}$ in Table~\ref{tb:gBM}.

\begin{table}[h!]
\begin{center}
\begin{tabular}{||c c c c c||} 
 \hline
 $\lambda$ & EM & SEM & TEM & ABEM \\ [0.5ex] 
 \hline\hline
 $2.5$ & $100/100$ & $100/100$ & $100/100$ & $100/100$ \\ 
 \hline
 $3$ & $99/100$ & $99/100$ & $99/100$ & $100/100$ \\ 
 \hline
 $3.5$ & $88/100$ & $89/100$ & $88/100$ & $100/100$ \\ [1ex]
 \hline
\end{tabular}
\caption{Proportion of samples containing only values in $D=(0,\infty)$ out of $100$ simulated sample paths for the Euler--Maruyama scheme (EM), the semi-implicit Euler--Maruyama scheme (SEM), the tamed Euler scheme (TE), and Artificial Barriers EM scheme (ABEM) for the SDE for geometric Brownian motion in~\eqref{eq:gBM} for different choices of $\lambda>0$ and parameters $T=1$, $x_{0}=0.1$, $d=1/2$ and $\Delta t = 10^{-2}$. \label{tb:gBM}}
\end{center}
\end{table}

Finally, we provide $L^{2}(\Omega)$-errors of the ABEM scheme in Figure~\ref{num:gBM_conv} and the numerical rate of $1/2$ confirms the theoretical rate obtained in Theorem~\ref{th:ABEMLpConv}. We used the parameters $T=1$, $x_{0}=0.1$, $d=1/2$ and the same values of $\lambda>0$ as in Table~\ref{tb:gBM} in Figure~\ref{num:gBM_conv}.

\begin{figure}[h!]
\begin{center}
  \includegraphics[width=0.8\textwidth]{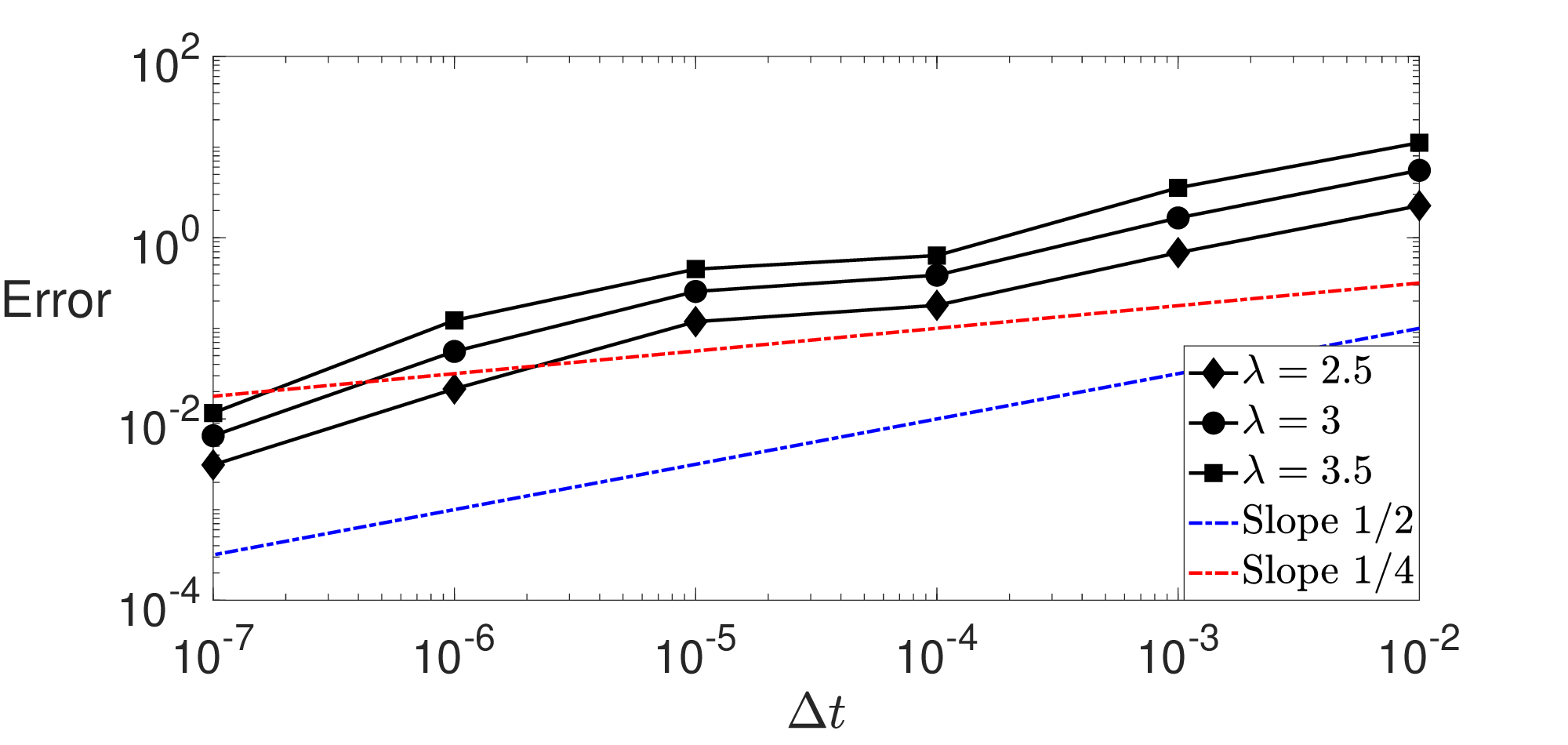}
  \caption{$L^{2}(\Omega)$-errors on the interval $[0,1]$ of the Artificial Barriers Euler--Maruyama scheme (ABEM) for the SDE for geometric Brownian motion in~\eqref{eq:gBM} for different choices of $\lambda>0$ and reference lines with slopes $1/2$ and $1/4$. Parameters: $T=1$, $x_{0}=0.1$, $d=1/2$ and $300$ Monte Carlo samples to approximate~\eqref{eq:L2err}.}\label{num:gBM_conv}
  \end{center}
\end{figure}

\subsection{Allen--Cahn type SDE}\label{subsec:ACsde}
Next we provide numerical experiments for the Allen--Cahn type SDE given by
\begin{equation}\label{eq:ACsde}
\diff X(t) = \left( X(t) - X(t)^{3} \right) \diff t + \lambda (1-X(t)^2) \diff B(t),\ t \in [0,T]
\end{equation}
with initial value $X(0)=x_{0} \in (-1,1)$; that is, $f(x) = x - x^{3}$ is cubic and $g(x) = \lambda(1 - x^{2})$ is quadratic in the considered SDE in equation~\eqref{eq:SDE}. For the same reason as in the example of geometric Brownian motion (gBm) in Section~\ref{subsec:GBM}, we use a parameter $\lambda>0$ in~\eqref{eq:ACsde}. In this case $D = (-1,1)$ and the Allen-Cahn type SDE in~\eqref{eq:ACsde} represents an instance of a two-finite-boundaries case. The considered Allen-Cahn type SDE is motivated, for example, by a spatial finite difference space discretisation of a stochastic Allen-Cahn PDE \cite{ALLEN19791085,MR3986273,MR1644183,MR3308418}. In contrast to the gBm case in Section~\ref{subsec:GBM}, there is no closed form solution to~\eqref{eq:ACsde}.

In Figure~\ref{num:AC_path_comp}, we provide sample paths for the EM, SEM, TE and ABEM schemes applied to the SDE~\eqref{eq:ACsde} that illustrate that the former three are not boundary-preserving while the ABEM scheme is boundary-preserving. We used the same Brownian motion sample path for all four schemes in Figure~\ref{num:AC_path_comp}. The sample paths of the EM, SEM and TE schemes almost coincide in Figure~\ref{num:AC_path_comp}.

\begin{figure}[h!]
\begin{center}
  \includegraphics[width=0.8\textwidth]{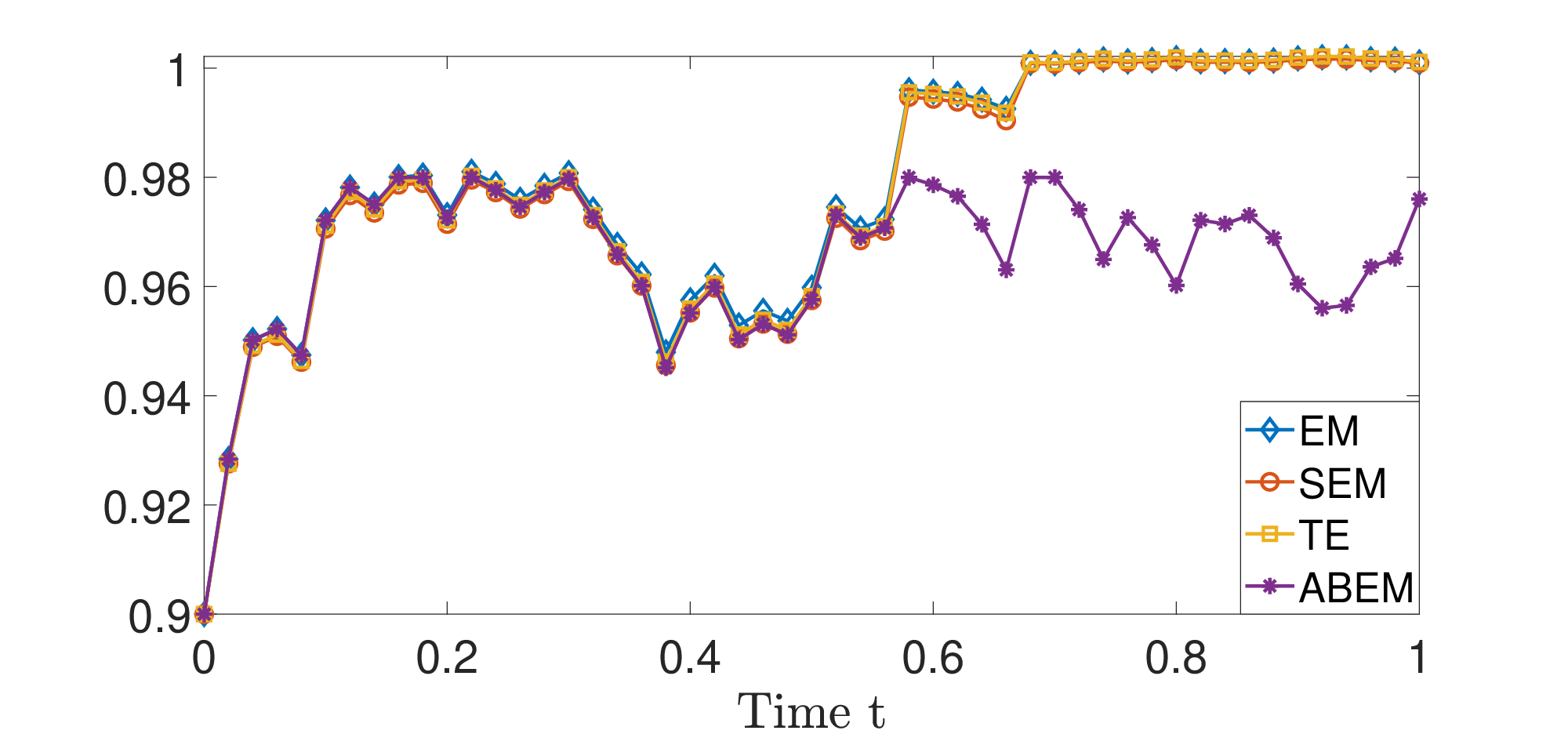}
  \caption{Path comparison of the Euler--Maruyama scheme (EM), the semi-implicit Euler--Maruyama scheme (SEM), the tamed Euler scheme (TE), and Artificial Barriers Euler--Maruyama scheme (ABEM) applied to the Allen-Cahn type SDE in~\eqref{eq:ACsde} with parameters $\lambda = 1$, $T = 1$, $x_{0}=0.9$, $d=1$ and $\Delta t = 0.02$.}\label{num:AC_path_comp}
  \end{center}
\end{figure}

As in Section~\ref{subsec:GBM}, we present in Table~\ref{tb:AC} the proportions out of $100$ samples of the EM, SEM, TE and ABEM schemes that preserved the domain $D = (-1,1)$ of the SDE in~\eqref{eq:ACsde}. The outcome of the experiment in Table~\ref{tb:AC} confirms that the ABEM scheme is boundary-preserving and that the EM, SEM and TE schemes are not. We used the parameters $T=1$, $x_{0}$ uniformly distributed on $(-1,1)$, $d=1/2$ and $\Delta t = 10^{-2}$ in Table~\ref{tb:AC}.

\begin{table}[h!]
\begin{center}
\begin{tabular}{||c c c c c||} 
 \hline
 $\lambda$ & EM & SEM & TEM & ABEM \\ [0.5ex] 
 \hline\hline
 $1$ & $100/100$ & $100/100$ & $100/100$ & $100/100$ \\ 
 \hline
 $1.5$ & $99/100$ & $99/100$ & $99/100$ & $100/100$ \\ 
 \hline
 $2$ & $59/100$ & $62/100$ & $65/100$ & $100/100$ \\ [1ex]
 \hline
\end{tabular}
\caption{Proportion of samples containing only values in $D = (-1,1)$ out of $100$ simulated sample paths for the Euler--Maruyama scheme (EM), the semi-implicit Euler--Maruyama scheme (SEM), the tamed Euler scheme (TE), and Artificial Barriers Euler--Maruyama scheme (ABEM) for the Allen--Cahn type SDE in~\eqref{eq:ACsde} for different choices of $\lambda>0$ and parameters $T=1$, $x_{0}$ uniformly distributed on $(-1,1)$ for each sample, $d = 1/2$ and $\Delta t = 10^{-2}$. \label{tb:AC}}
\end{center}
\end{table}

Lastly, we present in Figure~\ref{num:ACsde_conv} the $L^{2}(\Omega)$-errors of the ABEM scheme and the estimated rate of $1/2$ confirms the theoretical rate in Theorem~\ref{th:ABEMLpConv}. In Figure~\ref{num:AC_path_comp}, we used $T=1$, $x_{0}=0$, $d=1/2$,  and the same values of $\lambda>0$ as in Table~\ref{tb:AC} above.

\begin{figure}[h!]
\begin{center}
  \includegraphics[width=0.8\textwidth]{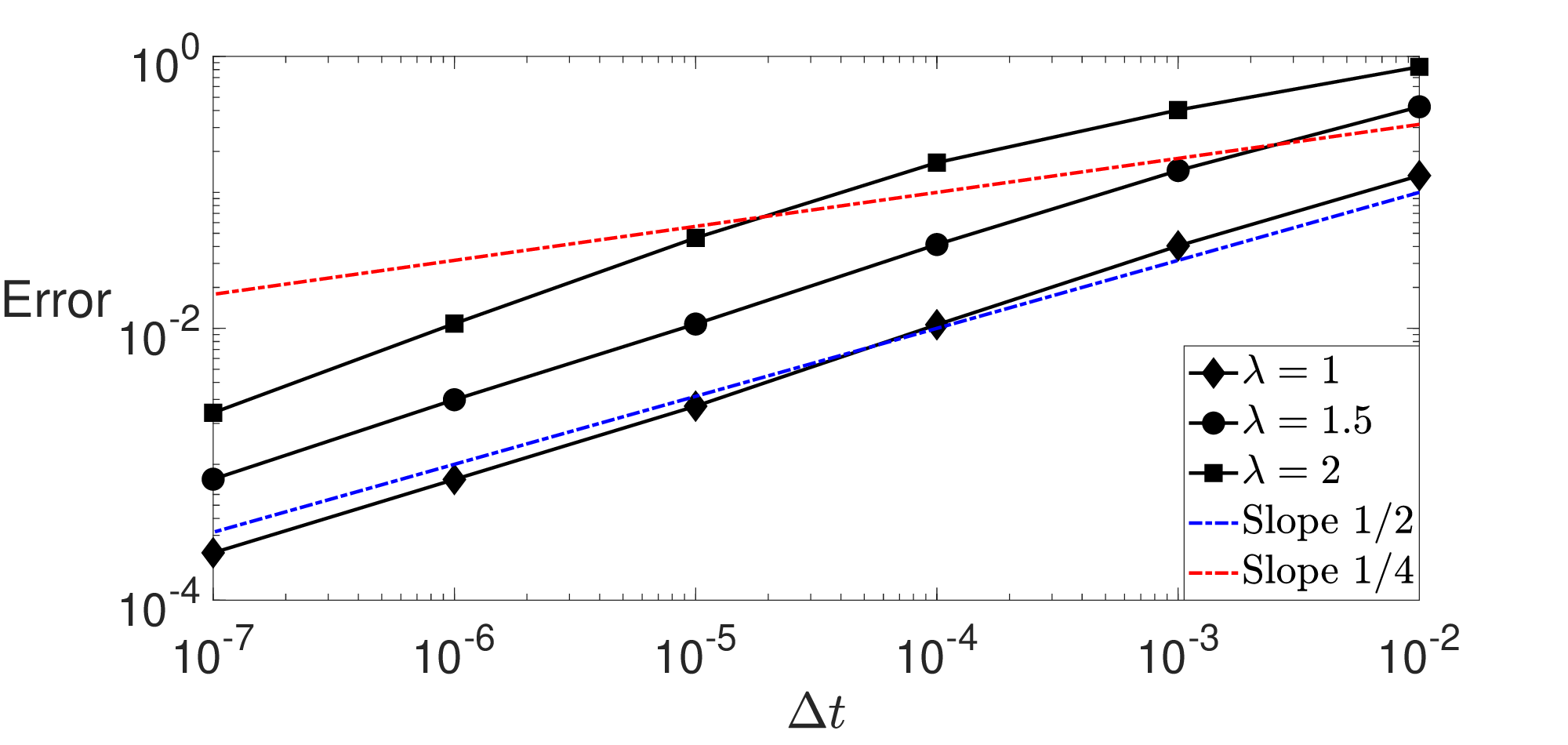}
  \caption{$L^{2}(\Omega)$-errors on the interval $[0,1]$ of the Artificial Barriers Euler--Maruyama scheme (ABEM) for the Allen--Cahn type SDE in~\eqref{eq:ACsde} for different choices of $\lambda>0$ and reference lines with slopes $1/2$ and $1/4$. Parameters: $T=1$, $x_{0}=0$, $d=1/2$ and $300$ Monte Carlo samples to approximate~\eqref{eq:L2err}.}\label{num:ACsde_conv}
  \end{center}
\end{figure}

\subsection{Hat function SDE}\label{subsec:hatFuncSDE}
Lastly, we provide numerical experiments for the hat function SDE given by
\begin{equation}\label{eq:hatSDE}
\diff X(t) = -H(X(t)) \diff t + \lambda H(X(t)) \diff B(t),\ t \in [0,T],
\end{equation}
with initial value $X(0)=x_{0} \in (0,1)$ and where $H:\mathbb{R} \to \mathbb{R}$ is the following hat function
\begin{equation*}
H(x) = 
\begin{cases} 
x,\ x \in [0,1/2], \\
1-x,\ x \in [1/2,1], \\
0,\ \text{else}.
\end{cases}
\end{equation*}
That is, $f$ and $g$ are piecewise linear and non-differentiable at $x=1/2$ in the considered SDE in equation~\eqref{eq:SDE}. As in Section~\ref{subsec:GBM} and in Section~\ref{subsec:ACsde}, we introduce a parameter $\lambda>0$. In this case $D = (0,1)$ and the hat function SDE in~\eqref{eq:hatSDE} is an example of a two-finite-boundaries case. We consider~\eqref{eq:hatSDE} as it is a case with non-differentiable coefficients functions where the Lamperti-transform cannot be applied. No closed form solution to~\eqref{eq:hatSDE} exists. 

First, we present sample paths in Figure~\ref{num:hatFunc_path_comp} that illustrate that the EM, SEM and TE schemes are not boundary-preserving while the ABEM scheme is boundary-preserving. All sample paths in Figure~\ref{num:hatFunc_path_comp} were computed using the same Brownian motion sample path. Note that, as $f$ and $g$ vanish outside the domain $D = (0,1)$, the EM, SEM and TE schemes stay constant after leaving the domain $D = (0,1)$. As was also the case in Section~\ref{subsec:GBM} and in Section~\ref{subsec:ACsde}, the sample paths for the EM, SEM and TE schemes almost coincide in Figure~\ref{num:hatFunc_path_comp}.

\begin{figure}[h!]
\begin{center}
  \includegraphics[width=0.8\textwidth]{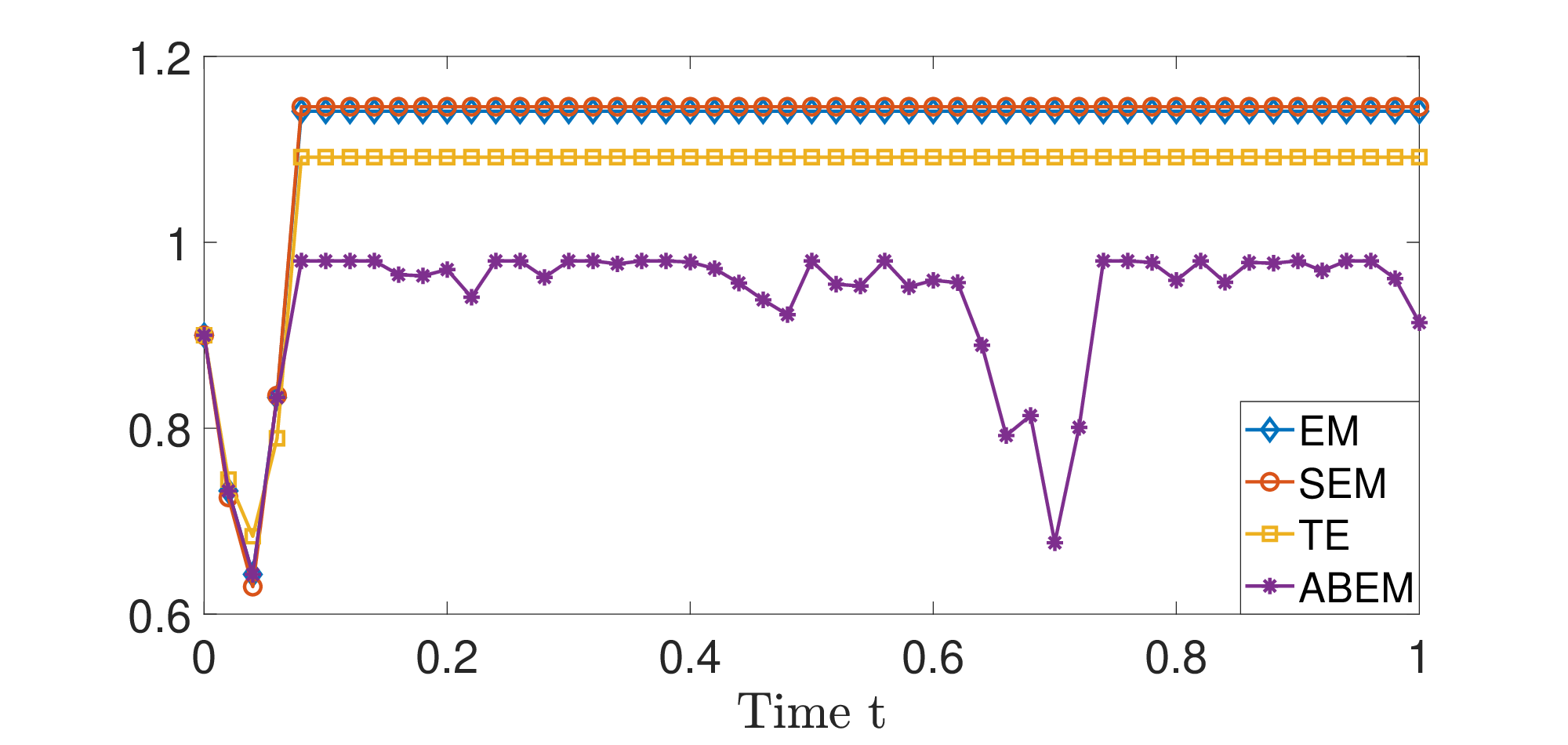}
  \caption{Path comparison of the EM, SEM, TE and ABEM schemes applied to the hat function SDE in~\eqref{eq:hatSDE} with parameters $\lambda = 3$, $T = 1$, $x_{0}=0.9$, $d=1$ and $\Delta t = 0.02$.}\label{num:hatFunc_path_comp}
  \end{center}
\end{figure}

Next, we consider numerical boundary-preservation of the EM, SEM, TE and ABEM schemes for the SDE~\eqref{eq:hatSDE}. We provide in Table~\ref{tb:hatFunc} the proportions out of $100$ generated samples that only contained values in $D = (0,1)$ of the EM, SEM, TE and ABEM schemes. Table~\ref{tb:hatFunc} confirms that the ABEM scheme is boundary-preserving while the EM, SEM and TE schemes are not.

\begin{table}[h!]
\begin{center}
\begin{tabular}{||c c c c c||} 
 \hline
 $\lambda$ & EM & SEM & TEM & ABEM \\ [0.5ex] 
 \hline\hline
 $1$ & $100/100$ & $100/100$ & $100/100$ & $100/100$ \\ 
 \hline
 $1.3$ & $99/100$ & $99/100$ & $99/100$ & $100/100$ \\ 
 \hline
 $1.6$ & $91/100$ & $93/100$ & $94/100$ & $100/100$ \\ [1ex]
 \hline
\end{tabular}
\caption{Proportion of samples containing only values in $D = (0,1)$ out of $100$ simulated sample paths for the Euler--Maruyama scheme (EM), the semi-implicit Euler--Maruyama scheme (SEM), the tamed Euler scheme (TE), and Artificial Barriers Euler--Maruyama scheme (ABEM) for the hat function SDE in~\eqref{eq:hatSDE} for different choices of $\lambda>0$, $T=1$, with $x_{0}$ uniformly distributed on $(0,1)$ for each sample, $d=1/2$ and $\Delta t = 10^{-2}$. \label{tb:hatFunc}}
\end{center}
\end{table}

Finally, we provide $L^{2}(\Omega)$-errors of the ABEM scheme applied to the SDE~\eqref{eq:hatSDE} in Figure~\ref{num:hatFunc_conv}. The estimated rate of $1/2$ in Figure~\ref{num:hatFunc_conv} confirms the theoretical rate obtained in Theorem~\ref{th:ABEMLpConv}. We used the parameters $T=1$, $x_{0} = 1/2$, $d=1/2$  and the same values of $\lambda>0$ as in Table~\ref{tb:hatFunc} in Figure~\ref{num:hatFunc_conv}.

\begin{figure}[h!]
\begin{center}
  \includegraphics[width=0.8\textwidth]{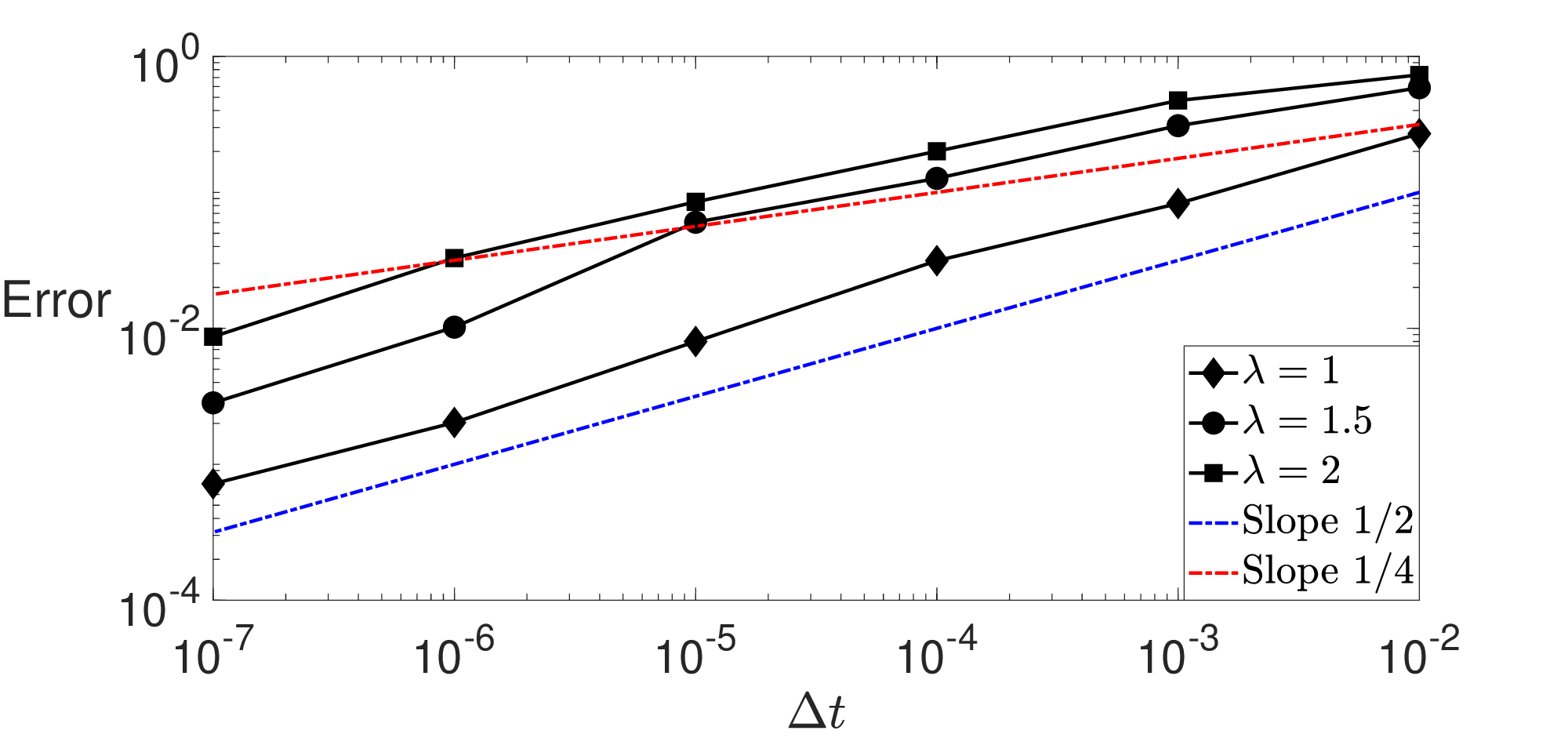}
  \caption{$L^{2}(\Omega)$-errors on the interval $[0,1]$ of the Artificial Barriers Euler--Maruyama scheme (ABEM) for the hat function SDE in~\eqref{eq:hatSDE} for different choices of $\lambda>0$ and reference lines with slopes $1/2$ and $1/4$. Parameters: $T=1$, $x_{0}=1/2$, $d=1/2$  and $300$ Monte Carlo samples to approximate~\eqref{eq:L2err}.}\label{num:hatFunc_conv}
  \end{center}
\end{figure}

\section*{Acknowledgements}
The author would like to thank Akash Sharma for reading and discussing the first version of the paper. The author would also like to thank David Cohen, Erik Jansson and Ioanna Motschan--Armen for their helpful comments to improve the presentation of parts of the work. This work is partially supported by the Swedish Research Council (VR) (David Cohen's project nr. $2018-04443$). 
The computations were enabled by resources provided by Chalmers e-Commons at Chalmers.

\bibliographystyle{abbrv}
\bibliography{Artificial_Barriers_for_stochastic_differential_equations_and_for_construction_of_Boundary-preserving_schemes}

\begin{thebibliography}{10}

\bibitem{MR3006996}
A.~Alfonsi.
\newblock Strong order one convergence of a drift implicit {E}uler scheme:
  application to the {CIR} process.
\newblock {\em Statist. Probab. Lett.}, 83(2):602--607, 2013.

\bibitem{ALLEN19791085}
S.~M. Allen and J.~W. Cahn.
\newblock A microscopic theory for antiphase boundary motion and its
  application to antiphase domain coarsening.
\newblock {\em Acta Metallurgica}, 27(6):1085--1095, 1979.

\bibitem{MR3363443}
F.~Black and M.~Scholes.
\newblock The pricing of options and corporate liabilities.
\newblock {\em J. Polit. Econ.}, 81(3):637--654, 1973.

\bibitem{MR4177372}
M.~Bossy, J.-F. Jabir, and K.~Mart\'inez.
\newblock On the weak convergence rate of an exponential {E}uler scheme for
  {SDE}s governed by coefficients with superlinear growth.
\newblock {\em Bernoulli}, 27(1):312--347, 2021.

\bibitem{Bossy2024}
M.~Bossy and K.~Martínez.
\newblock Strong convergence of the exponential euler scheme for sdes with
  superlinear growth coefficients and one-sided lipschitz drift.
\newblock {\em Preprint, arXiv:2405.00806}, 2024.

\bibitem{MR4780408}
C.-E. Br\'ehier, D.~Cohen, and J.~Ulander.
\newblock Analysis of a positivity-preserving splitting scheme for some
  semilinear stochastic heat equations.
\newblock {\em ESAIM Math. Model. Numer. Anal.}, 58(4):1317--1346, 2024.

\bibitem{MR4729657}
C.-E. Br\'ehier, D.~Cohen, and J.~Ulander.
\newblock Positivity-preserving schemes for some nonlinear stochastic {PDE}s.
\newblock In {\em Sixteenth {I}nternational {C}onference {Z}aragoza-{P}au on
  {M}athematics and its {A}pplications}, volume~43 of {\em Monogr. Mat.
  Garc\'ia Galdeano}, pages 31--40. Prensas Univ. Zaragoza, Zaragoza, 2024.

\bibitem{MR3986273}
C.-E. Br\'{e}hier and L.~Gouden\`ege.
\newblock Analysis of some splitting schemes for the stochastic {A}llen-{C}ahn
  equation.
\newblock {\em Discrete Contin. Dyn. Syst. Ser. B}, 24(8):4169--4190, 2019.

\bibitem{MR4220738}
L.~Chen, S.~Gan, and X.~Wang.
\newblock First order strong convergence of an explicit scheme for the
  stochastic {SIS} epidemic model.
\newblock {\em J. Comput. Appl. Math.}, 392:Paper No. 113482, 16, 2021.

\bibitem{Chitashvili}
R.~Chitashvili and N.~Lazrieva.
\newblock Strong solutions of stochastic differential equations with boundary
  conditions.
\newblock {\em Stochastics}, 5(4):255--309, 1981.

\bibitem{MR2931351}
C.~E. Dangerfield, D.~Kay, S.~MacNamara, and K.~Burrage.
\newblock A boundary preserving numerical algorithm for the {W}right-{F}isher
  model with mutation.
\newblock {\em BIT}, 52(2):283--304, 2012.

\bibitem{MR2898556}
S.~Dereich, A.~Neuenkirch, and L.~Szpruch.
\newblock An {E}uler-type method for the strong approximation of the
  {C}ox-{I}ngersoll-{R}oss process.
\newblock {\em Proc. R. Soc. Lond. Ser. A Math. Phys. Eng. Sci.},
  468(2140):1105--1115, 2012.

\bibitem{MR2416011}
D.~Ding and Y.~Y. Zhang.
\newblock A splitting-step algorithm for reflected stochastic differential
  equations in {$\Bbb R^1_+$}.
\newblock {\em Comput. Math. Appl.}, 55(11):2413--2425, 2008.

\bibitem{MR1110990}
P.~Dupuis and H.~Ishii.
\newblock On {L}ipschitz continuity of the solution mapping to the {S}korokhod
  problem, with applications.
\newblock {\em Stochastics Stochastics Rep.}, 35(1):31--62, 1991.

\bibitem{GraySIS}
A.~Gray, D.~Greenhalgh, L.~Hu, X.~Mao, and J.~Pan.
\newblock A stochastic differential equation {SIS} epidemic model.
\newblock {\em SIAM Journal on Applied Mathematics}, 71(3):876--902, 2011.

\bibitem{MR1644183}
I.~Gy\"ongy.
\newblock Lattice approximations for stochastic quasi-linear parabolic partial
  differential equations driven by space-time white noise. {I}.
\newblock {\em Potential Anal.}, 9(1):1--25, 1998.

\bibitem{MR3433041}
N.~Halidias.
\newblock Constructing positivity preserving numerical schemes for the
  two-factor {CIR} model.
\newblock {\em Monte Carlo Methods Appl.}, 21(4):313--323, 2015.

\bibitem{MR3331648}
N.~Halidias.
\newblock Construction of positivity preserving numerical schemes for some
  multidimensional stochastic differential equations.
\newblock {\em Discrete Contin. Dyn. Syst. Ser. B}, 20(1):153--160, 2015.

\bibitem{MR2985171}
M.~Hutzenthaler, A.~Jentzen, and P.~E. Kloeden.
\newblock Strong convergence of an explicit numerical method for {SDE}s with
  nonglobally {L}ipschitz continuous coefficients.
\newblock {\em Ann. Appl. Probab.}, 22(4):1611--1641, 2012.

\bibitem{Karlin1981ASC}
S.~Karlin and H.~E. Taylor.
\newblock {\em A second course in stochastic processes}.
\newblock Academic Press, 1981.

\bibitem{MR4544037}
C.~Kelly and G.~J. Lord.
\newblock An adaptive splitting method for the {C}ox-{I}ngersoll-{R}oss
  process.
\newblock {\em Appl. Numer. Math.}, 186:252--273, 2023.

\bibitem{KERMACK199133}
W.~Kermack and A.~McKendrick.
\newblock Contributions to the mathematical theory of epidemics—i.
\newblock {\em Bulletin of Mathematical Biology}, 53(1):33--55, 1991.

\bibitem{DomPres}
Y.~Kiouvrekis and I.~S. Stamatiou.
\newblock Domain preserving and strongly converging explicit scheme for the
  stochastic {SIS} epidemic model.
\newblock {\em Preprint, arXiv:2307.14404}, 2023.

\bibitem{introTostocCalc}
F.~C. Klebaner.
\newblock {\em Introduction to Stochastic Calculus with Applications}.
\newblock Imperial Collage Press, 3rd edition, 2012.

\bibitem{MR2320830}
P.~E. Kloeden and A.~Neuenkirch.
\newblock The pathwise convergence of approximation schemes for stochastic
  differential equations.
\newblock {\em LMS J. Comput. Math.}, 10:235--253, 2007.

\bibitem{MR1214374}
P.~E. Kloeden and E.~Platen.
\newblock {\em Numerical solution of stochastic differential equations},
  volume~23 of {\em Applications of Mathematics (New York)}.
\newblock Springer-Verlag, Berlin, 1992.

\bibitem{MR2349573}
L.~Kruk, J.~Lehoczky, K.~Ramanan, and S.~Shreve.
\newblock An explicit formula for the {S}korokhod map on {$[0,a]$}.
\newblock {\em Ann. Probab.}, 35(5):1740--1768, 2007.

\bibitem{MR4475995}
Z.~Lei, S.~Gan, and Z.~Chen.
\newblock Strong and weak convergence rates of logarithmic transformed
  truncated {EM} methods for {SDE}s with positive solutions.
\newblock {\em J. Comput. Appl. Math.}, 419:Paper No. 114758, 21, 2023.

\bibitem{MR1341164}
D.~L\'epingle.
\newblock Euler scheme for reflected stochastic differential equations.
\newblock volume~38, pages 119--126. 1995.
\newblock Probabilit\'es num\'eriques (Paris, 1992).

\bibitem{MR4888024}
R.~Liu, A.~Neuenkirch, and X.~Wang.
\newblock A strong order {$1.5$} boundary preserving discretization scheme for
  scalar {SDE}s defined in a domain.
\newblock {\em Math. Comp.}, 94(354):1815--1862, 2025.

\bibitem{MR4597411}
R.~Liu and X.~Wang.
\newblock A higher order positivity preserving scheme for the strong
  approximations of a stochastic epidemic model.
\newblock {\em Commun. Nonlinear Sci. Numer. Simul.}, 124:Paper No. 107258, 23,
  2023.

\bibitem{MR2689982}
Y.~Liu.
\newblock {\em Numerical approaches to stochastic differential equations with
  boundary conditions}.
\newblock ProQuest LLC, Ann Arbor, MI, 1993.
\newblock Thesis (Ph.D.)--Purdue University.

\bibitem{MR1341162}
Y.~Liu.
\newblock Discretization of a class of reflected diffusion processes.
\newblock volume~38, pages 103--108. 1995.
\newblock Probabilit\'es num\'eriques (Paris, 1992).

\bibitem{MR3308418}
G.~J. Lord, C.~E. Powell, and T.~Shardlow.
\newblock {\em An introduction to computational stochastic {PDE}s}.
\newblock Cambridge Texts in Applied Mathematics. Cambridge University Press,
  New York, 2014.

\bibitem{MR4242953}
X.~Mao, F.~Wei, and T.~Wiriyakraikul.
\newblock Positivity preserving truncated {E}uler-{M}aruyama method for
  stochastic {L}otka-{V}olterra competition model.
\newblock {\em J. Comput. Appl. Math.}, 394:Paper No. 113566, 17, 2021.

\bibitem{MR496534}
R.~C. Merton.
\newblock Theory of rational option pricing.
\newblock {\em Bell J. Econom. and Management Sci.}, 4:141--183, 1973.

\bibitem{Mller2010FromSD}
J.~K. M{\o}ller and H.~Madsen.
\newblock From state dependent diffusion to constant diffusion in stochastic
  differential equations by the lamperti transform.
\newblock Technical report, Technical University of Denmark, DTU Informatics,
  Building 321. IMM-Technical Report-2010-16, 2010.

\bibitem{MR2341800}
E.~Moro and H.~Schurz.
\newblock Boundary preserving semianalytic numerical algorithms for stochastic
  differential equations.
\newblock {\em SIAM J. Sci. Comput.}, 29(4):1525--1549, 2007.

\bibitem{MR3248050}
A.~Neuenkirch and L.~Szpruch.
\newblock First order strong approximations of scalar {SDE}s defined in a
  domain.
\newblock {\em Numer. Math.}, 128(1):103--136, 2014.

\bibitem{MR2001996}
B.~{\O}ksendal.
\newblock {\em Stochastic differential equations}.
\newblock Universitext. Springer-Verlag, Berlin, sixth edition, 2003.
\newblock An introduction with applications.

\bibitem{MR1357657}
R.~Pettersson.
\newblock Approximations for stochastic differential equations with reflecting
  convex boundaries.
\newblock {\em Stochastic Process. Appl.}, 59(2):295--308, 1995.

\bibitem{Pilipenko2014AnIT}
A.~Pilipenko.
\newblock {\em An introduction to stochastic differential equations with
  reflection}, volume~1 of {\em Lectures in Pure and Applied Mathematics}.
\newblock Potsdam University Press, 09 2014.

\bibitem{MR3543890}
S.~Sabanis.
\newblock Euler approximations with varying coefficients: the case of
  superlinearly growing diffusion coefficients.
\newblock {\em Ann. Appl. Probab.}, 26(4):2083--2105, 2016.

\bibitem{MR873889}
Y.~Saisho.
\newblock Stochastic differential equations for multidimensional domain with
  reflecting boundary.
\newblock {\em Probab. Theory Related Fields}, 74(3):455--477, 1987.

\bibitem{MR1410392}
H.~Schurz.
\newblock Numerical regularization for {SDE}s: construction of nonnegative
  solutions.
\newblock {\em Dynam. Systems Appl.}, 5(3):323--351, 1996.

\bibitem{MR4521003}
A.~V. Skorokhod.
\newblock Stochastic equations for diffusion processes in a bounded region.
\newblock {\em Theory of Probability and Its Applications}, 6(3):264--274,
  1961.

\bibitem{MR1840835}
L.~S\l{}omi\'nski.
\newblock Euler's approximations of solutions of {SDE}s with reflecting
  boundary.
\newblock {\em Stochastic Process. Appl.}, 94(2):317--337, 2001.

\bibitem{SLOMINSKI1994197}
L.~Słomiński.
\newblock On approximation of solutions of multidimensional {SDE}'s with
  reflecting boundary conditions.
\newblock {\em Stochastic Processes and their Applications}, 50(2):197--219,
  1994.

\bibitem{MR529332}
H.~Tanaka.
\newblock Stochastic differential equations with reflecting boundary condition
  in convex regions.
\newblock {\em Hiroshima Math. J.}, 9(1):163--177, 1979.

\bibitem{MR4737060}
J.~Ulander.
\newblock Boundary-preserving {L}amperti-splitting schemes for some stochastic
  differential equations.
\newblock {\em J. Comput. Dyn.}, 11(3):289--317, 2024.

\bibitem{UweakBP}
J.~Ulander.
\newblock Boundary-preserving weak approximations of some semilinear stochastic
  partial differential equations.
\newblock {\em Preprint, arXiv:2412.10800}, 2024.

\bibitem{MR4274899}
H.~Yang and J.~Huang.
\newblock First order strong convergence of positivity preserving logarithmic
  {E}uler-{M}aruyama method for the stochastic {SIS} epidemic model.
\newblock {\em Appl. Math. Lett.}, 121:Paper No. 107451, 7, 2021.

\end{thebibliography}

%\end{sloppypar}
\end{document}